%% file: linearity_Reidemeister.tex
\documentclass{amsart}
\input{decls}
\usepackage{tikz}
\usepackage{mathdots}
\usepackage{wasysym}

\newcommand{\tr}{\ensuremath{\operatorname{tr}}}

\newcommand{\dual}[1]{D{#1}}
\newcommand{\rdual}[1]{D_r{#1}}

\autodefs{\bProf\cProf\bSet\nAut\bCat\lMod\lMat\lProf}
\let\D\sD

\theoremstyle{plain}
\makeatletter
\newtheorem*{rep@theorem}{\rep@title}
\newcommand{\newreptheorem}[2]{%
\newenvironment{rep#1}[1]{%
 \def\rep@title{#2 \ref{##1}}%
 \begin{rep@theorem}}%
 {\end{rep@theorem}}}
\makeatother

\newreptheorem{thm}{Theorem}

\let\prof\bProf

\let\dV\sV
\let\dW\sW
\let\dT\sT
\let\dprof\cProf

\autodefs{\cDER\cCat\cCAT\cMONCAT\bSet\cProf\bCat\cPro\cMONDER\cBICAT\ncomp\nid\sProf\ncyl\fC\fZ\fT\nhocolim\nholim\cPsNat\ncoll\bTop\cRMODEL\cEx\cSp}
\def\ho{\mathscr{H}\!\mathit{o}\xspace}

\newcommand{\ptop}{\sT}
\newcommand{\tc}{\bbone}          % Terminal category (\bbone)
\newcommand{\twc}{\tc\sqcup\tc}           % Two-object category (\bbtwo)
\newcommand{\pt}{\star} 

\renewcommand{\fib}[2]{{#1}^{-1}(#2)}

\def\calDi#1#2{\ensuremath{\mathscr{#1}\!/\!_{\mathbf{#2}}}}

\newcommand{\pb}[1]{{#1}^{\ast}}%{\blacktriangleleft}}
\newcommand{\pf}[1]{{#1}_{!}}%\blacktriangleright}}
\newcommand{\copf}[1]{{#1}_{\ast}}%\blacktriangleright}}
\newcommand{\xrhd}{\mathrel{\overline{\rhd}}}
\newcommand{\Ex}{\cEx} 

\newcommand{\ev}[1]{\operatorname{ev}_{#1}}
\renewcommand{\lan}[1]{\operatorname{Lan}_{#1}}

\def\fn{g}
\newcommand{\exsm}{\mathbin{\overline{\sm}}}
\newcommand{\bbZ}{\mathbb{Z}}

\title{The linearity of fixed point invariants}
\author{Kate Ponto and Michael Shulman}
\date{\today}
\thanks{The first author  was partially 
supported by NSF grant DMS-1207670. The second author was partially supported by an NSF postdoctoral fellowship and 
NSF grant DMS-1128155, and appreciates the hospitality of the University of Kentucky.
Any opinions, findings, and conclusions or recommendations expressed in this material are those of the authors and 
do not necessarily reflect the views of the National Science Foundation.}
\address{University of Kentucky\\719 Patterson Office Tower\\Lexington, KY 40508}
\email{kate.ponto@uky.edu}
\address{University of San Diego\\5998 Alcala Park\\San Diego, CA 92110}
\email{shulman@sandiego.edu}

\begin{document}
\maketitle

\begin{abstract}
  We prove two general decomposition theorems for fixed-point invar\-iants: one for the Lefschetz number and one for the Reidemeister trace.
  These theorems imply the familiar additivity results for these invariants.
  Moreover, the proofs of these theorems are essentially formal, taking place in the abstract context of bicategorical traces.
  This makes it straightforward to generalize the theory to analogous invariants in other contexts, such as equivariant and fiberwise homotopy theory.
\end{abstract}

\tableofcontents

\section{Introduction}
\label{sec:introduction}

In topological fixed-point theory, one of the most basic invariants is the \emph{Lefschetz number} $L(f)$ of an endomorphism $f$ of a finite CW complex $X$, which vanishes whenever $f$ is homotopic to a map with no fixed points.
If $f$ is the identity map, then $L(f)$ is the \emph{Euler characteristic} $\chi(X)$ of the complex $X$.

Two of the most important tools for calculating Lefschetz numbers (and hence also Euler characteristics) are \emph{additivity} and \emph{multiplicativity}.
The classical additivity theorem states that if $X\subset Y$ is a subcomplex and $f\colon Y\rightarrow Y$ takes $X$ into itself, then
\begin{equation}
  L(f)=L(f|_X)+L(f/X)\label{eq:intro-add}
\end{equation}
where $f|_X$ denotes the restriction of $f$ to $X$ and $f/X$ denotes the induced map on the quotient $Y/X$.
The classical multiplicativity theorem is that if $p\colon E\xto{}B$ is a fibration with fiber $F$ and simply connected base $B$, and $f\colon E\xto{}E$ and $f_B\colon B\xto{}B$ are endomorphisms satisfying $p\circ f=f_B\circ p$, then
\begin{equation}
  L(f)=L(f_B)\cdot L(f|_F).\label{eq:intro-mult}
\end{equation}

%There are many refinements and extensions of these results.  
%For example, Ferrario  \cite[Theorem~3.2.1]{ferrario} showed that
% if $X\subset Y$ is a subcomplex and $f\colon Y\rightarrow Y$ takes $X$ into itself, then
%\[  j(R(f))- L(f|_A)= R(f/A)\]
%where $R(f)$ is the Reidemeister trace (which refines the Lefschetz number with more information appropriate to non-simply-connected spaces) of $f$ and $j\colon X\to X/A$ is the quotient map.

This paper is part of a project to prove similar theorems for generalizations of the Lefschetz number, such as the Reidemeister trace (which refines the Lefschetz number with information appropriate to non-simply-connected spaces), as well as versions of the Lefschetz number and Reidemeister trace for parametrized or equivariant endomorphisms \cite{kate:traces, kate:equiv}.
These refined invariants are relatively incompatible with the classical approaches to additivity and multiplicativity, so we instead approach these results by embracing a greater level of abstraction. 

In \cite{dp:duality}, Dold and Puppe observed that the Lefschetz number is an example of the \emph{trace} in a symmetric monoidal category.  This generalizes the familiar trace for linear transformations.
(See~\cite{PS1} for an overview of traces in symmetric monoidal categories.)
They also noted that the trace in an equivariant or parameterized stable homotopy category is the equivariant or parameterized Lefschetz number, and that the Lefschetz fixed-point theorem is a consequence of the fact that symmetric monoidal functors preserve traces.

In~\cite{kate:traces,PS2}, the first author observed that the Reidemeister trace can be described in a similar way %as well as additivity and multipliciativty results,
by generalizing from symmetric monoidal traces to \emph{bicategorical traces}.
Here a bicategory is regarded as a ``many-object'' version of a monoidal category, where the objects are (in the simplest version) groups or rings that encode the actions of fundamental groups. 
This generalization yields definitions of equivariant and parameterized Reidemeister traces, and the Reidemeister-trace version of the Lefschetz fixed-point theorem follows from the preservation of bicategorical traces by suitable functors of bicategories.

From this perspective,~\eqref{eq:intro-add} and~\eqref{eq:intro-mult} have very suggestive parallels with the trace of a linear transformation: 
the additivity of the Lefschetz number parallels the additivity of the trace of linear transformations on short exact sequences, and the multiplicativity of the Lefschetz number parallels the compatibility between the 
trace and tensor product of linear transformations.
This suggests that additivity and multiplicativity results for fixed-point invariants can usefully be situated in the larger project of establishing such results for the symmetric monoidal 
and bicategorical traces.

In fact, two special cases of additivity and multiplicativity are immediate consequences of the identification of fixed-point invariants as traces.
Traces in a closed symmetric monoidal category are automatically compatible with coproducts and with the monoidal product itself, yielding respectively the ``trivial'' case of~\eqref{eq:intro-add} when $X$ is the inclusion of a disjoint summand and the ``trivial'' case of~\eqref{eq:intro-mult} when $E=B\times F$ and $p$ is the product projection.

In~\cite{add} May extended these compatibility arguments to the general case of~\eqref{eq:intro-add} for Lefschetz numbers, using triangulated categories with compatible monoidal structures, and in~\cite{gps:additivity} we reformulated his proof using stable derivators.
The proof is much more complicated than that of the trivial case, but the basic idea is the same: the quotient $Y/X$ is a kind of colimit, and the tensor product preserves colimits in each variable, yielding a compatibility relation for traces from which~\eqref{eq:intro-add} can be extracted.

By contrast, the nontrivial case of multiplicativity requires a change of perspective.
The hypothesis of simple connectivity is unnatural from an abstract point of view.
But in its absence, the analogue of~\eqref{eq:intro-mult} must involve not only the Lefschetz number of $f_B$ but its Reidemeister trace, and this pushes us into the world of bicategorical trace.
This led us in~\cite{PS3} to study traces in bicategories that can contain arbitrary fibrations $p:E\to B$, such as those of parametrized spaces and spectra from~\cite{maysig:pht}.

Then in~\cite{PS4} we gave an abstract proof of multiplicativity: not only~\eqref{eq:intro-mult} but also a more refined version that eliminates the hypothesis of simple connectivity, as well as a corresponding result for the Reidemeister trace, all of which can easily be generalized to other contexts including the parametrized and equivariant worlds.
The idea is to view the multiplication in~\eqref{eq:intro-mult} as the product of two $1\times 1$ matrices, and then generalize it to the product of a $1\times n$ matrix and an $n\times 1$ matrix.
The dimension $n$ encodes $\pi_1(B)$, and the matrix product is a ``decategorification'' of the composition in a bicategory of parametrized spaces or spectra.

In this paper and its companion~\cite{PS5}, we apply a similar perspective to prove the additivity formula~\eqref{eq:intro-add} and significant generalizations thereof.
The most familiar of these generalizations is the following additivity formula for the Reidemeister trace.

\begin{thm}\label{thm:addreidemeisteri}If $X$ and $Y$ are closed smooth manifolds with 
$X\subset Y$, $f\colon Y\to Y$ is a continuous map such that $f(X)\subset X$, $\Lambda^fY$ is the twisted loop space,  and $i\colon \Lambda^{f|_X} X\xto{}\Lambda^f Y$  is the inclusion, then
\[R(f)-i(R(f|_X))=R_{Y|X}(f).\]  
The invariant $R_{Y|X}(f)$ is the \emph{relative
Reidemeister trace} of $f$ \cite{kate:relative}.  If $X\subset Y$ is a cofibration, $R_{Y|X}(f)$ is a refinement of the Reidemeister trace of $f/X\colon Y/X\to Y/X$ that takes values in $\pi_0^s(\Lambda^f(Y))$.
\end{thm}

While this particular result is a small generalization of \cite[Theorem~3.2.1]{ferrario}, our method is dramatically different from Ferrario's.
Most significantly, our method can be applied directly to produce the analogous results for parametrized and equivariant Lefschetz numbers and Reidemeister traces.

The insight leading to our proof of additivity is to rewrite~\eqref{eq:intro-add} as
\begin{equation}
  L(f/X) = L(f) - L(f|_X). \label{eq:intro-add2}
\end{equation}
In this form, it can be viewed as a calculation of the invariant associated to a \emph{colimit}, namely the quotient $Y/X$, in terms of the corresponding invariants associated to the input data ($X$ and $Y$).
We now view the right-hand side of~\eqref{eq:intro-add2} as the product of the $1\times 2$ matrix $[L(f), L(f|_X)]$ and the $2\times 1$ matrix $\scriptsize\left[
\begin{array}{c}
  1 \\ -1
\end{array}\right]
$.
The dimension $2$ and the latter matrix are invariants determined by the type of colimit under consideration, and the matrix product is a decategorification of the composition in a suitable bicategory (namely, a bicategory of profunctors).
In~\cite{PS5} we develop this idea into a general method for calculating trace invariants associated to colimits of diagrams; while in this paper we specialize these results to diagrams of spaces, recovering~\eqref{eq:intro-add2} by specializing further to quotients.

There is a strong similarity to the results of~\cite{PS4} in both perspective and technique: in both cases we build on a generalization of the invariance of trace under change of basis that is proven by abstract diagram chases, and we describe the trace of a morphism in terms of traces of constituent pieces.
And indeed, from an $\infty$-categorical point of view, the total space of a fibration can be identified with a colimit whose ``diagram shape'' is the base space regarded as an $\infty$-groupoid.

However, because we take a more concrete topological approach, there are also major technical differences between the two theories.
The main one is that when dealing with \emph{diagrams} in homotopy theory (which we did not have to do in~\cite{PS4}), there is a constant need for \emph{fibrant and cofibrant replacement} in order to ensure that various maps that ought to exist really do exist.
To avoid this complication, we use a technical device called a \emph{derivator bicategory}, which encapsulates all the fibrant and cofibrant replacement in an abstract structure so that we never have to think about it again.

Although this abstract structure is very convenient for proving the abstract theorems, the topologically inclined reader will naturally want a more concrete description of what is going on.
This is the reason for the division between this paper and its companion~\cite{PS5}:
in~\cite{PS5} we develop the general theory abstractly, while in this paper we apply that theory more concretely to the cases of topological interest.
Formally speaking, the latter simply means verifying that the topological context gives rise to a derivator bicategory; we will do that in \autoref{part:derivators} of this paper.
However, since many readers may not be interested in that level of abstraction, in \autoref{part:topological} we give a more concrete translation of the theorems in purely topological language.
We encourage the reader with a purely topological interest to focus on \autoref{part:topological} and treat \autoref{part:derivators} and~\cite{PS5} as ``black boxes''.
Although we sketch some proofs in \autoref{part:topological} to give an idea of what is going on, formally speaking the real work is done in~\cite{PS5} and \autoref{part:derivators}.

We begin \autoref{part:topological} in \S\ref{sec:lefschetz-number} by explaining how to view the Lefschetz number as a trace, and the implications of this viewpoint (for instance, it immediately implies multiplicativity for trivial bundles).
In \S\ref{sec:lefschetz} we then generalize this to traces of \emph{diagrams}, such as the map $X\to Y$ of which we considered the quotient above.
The invariant associated to an endomorphism of a diagram includes the Lefschetz number of all the induced endomorphisms of the component spaces.
The central theorem then says that the Lefschetz number of the colimit (such as $Y/X$) can be calculated as a linear combination of these component Lefschetz numbers (such as~\eqref{eq:intro-add2}); thus we call it a \emph{linearity formula}.

In \S\ref{sec:fibration} we move on to the Reidemeister trace, explaining how to view it as a trace.
This involves introducing \emph{parametrized spaces} as in~\cite{maysig:pht}, but we try to remain as explicit as possible.
%(In~\cite{PS4} we used this viewpoint on the Reidemeister trace to prove the general multiplicativity theorem.)
Finally, in \S\ref{sec:parametrize} we combine \S\ref{sec:lefschetz} and \S\ref{sec:fibration} to deduce a linearity formula for traces of diagrams of parametrized spaces, and in \S\ref{sec:reidemeister} we use this theory to deduce the additivity of Reidemeister trace.

\autoref{part:derivators} begins in \S\ref{sec:deriv-bicat} by recalling the definition of a derivator bicategory from \cite{PS5} and the statements of some of the results we used in \S\ref{sec:parametrize} in the greater generality of a derivator bicategory.
Then in \S\ref{sec:indexed} we describe a naturally occurring structure on parametrized spaces that we use to construct the derivator bicategory of parametrized spectra, generalizing \emph{indexed monoidal categories} (see \cite{PS3}) to \emph{indexed monoidal derivators}.
In \S\ref{sec:indclosed} we describe closed structures on indexed monoidal derivators, and hence on the derivator bicategories they give rise to; these are used in \cite{PS5} to make defining and constructing dual pairs easier.
Finally, in \S\ref{sec:indexed-mmc} we explain how to construct an indexed monoidal derivator from an \emph{indexed monoidal model category}.
This allows us to produce the desired derivator bicategory of parametrized spectra, starting from an indexed monoidal model category of parametrized spectra (which is essentially what was constructed in~\cite{maysig:pht}).

\part{Topological fixed-point invariants}
\label{part:topological}

In these sections we give an explicit, topological description of the linearity results for fixed point theory to the extent that it is illuminating.  
In practice, the approach we are using is clear when working in this specific context, but there are a few important results whose proofs become 
unpleasant and relatively intractable when interpreted in particular examples and are far more manageable in greater generality.  For these results
we provide an indication of the proof, but leave the formal proof to \autoref{part:derivators} and \cite{PS5}.

It is important to note that while we state the results in this section in terms of classical stable homotopy theory, they apply 
in equivariant and fiberwise stable homotopy as well.   Generalizing further, the approach of \S\ref{sec:lefschetz-number}  and \S\ref{sec:lefschetz}  applies
in any ``homotopy category'', such as those arising from model categories, and the approach of \S\ref{sec:fibration} to \S\ref{sec:reidemeister} applies to ``indexed monoidal 
model categories'' as defined in  \S\ref{sec:indexed-mmc}.

\section{The Lefschetz number as a trace}
\label{sec:lefschetz-number}

We begin by recalling how the Lefschetz number can be regarded as a trace.
Let $M$ be a based topological space.
We say that $M$ is \textbf{$n$-dualizable} if there is a based space $\dual{M}$ and maps
\begin{align*}
  \eta &\colon S^n \too M \sm \dual{M}\\
  \ep &\colon \dual{M} \sm M \too S^n
\end{align*}
such that the composites
\begin{gather}
  S^n \sm M \xto{\eta\sm \id} M\sm \dual{M}\sm M \xto{\id \sm \ep}  M \sm S^n \label{eq:tri1} \\
  \dual{M}\sm S^n \xto{\id\sm \eta} \dual{M}\sm M\sm\dual{M} \xto{ \ep\sm \id} S^n \sm \dual{M}\label{eq:tri2}
\end{gather}
become homotopic to transposition maps after smashing with some $S^m$ (in this case one says they are \emph{stably homotopic} to transpositions).
We refer to $\eta$ as the \textbf{coevaluation} and $\ep$ as the \textbf{evaluation} for the duality.

An unbased space $M$ is \textbf{$n$-dualizable} if $M$ with a disjoint basepoint, written $M_+$, is $n$-dualizable in the above sense.
It is well-known (see~\cite{atiyah:thom, lms:equivariant}) that every closed smooth manifold $M$ is $n$-dualizable.
We may take $n$ to be the dimension of a Euclidean space in which $M$ embeds, and $\dual{M}$ the Thom space of the normal bundle of the embedding.

If $M$ is an $n$-dualizable based space and $f\colon M\to M$ is a based endomorphism, we define its \textbf{trace} $\tr(f)$ to be the composite map
\begin{equation}
  S^n \xto{\eta} M \sm \dual{M} \xto{f \sm \id} M \sm \dual{M} \xto{\cong} \dual{M}\sm M \xto{\ep} S^n\label{eq:tr}
\end{equation}
and its \textbf{Lefschetz number} $L(f)$ to be the degree of this trace.
The \textbf{Euler characteristic} of $M$, $\chi(M)$, is the Lefschetz number of the identity map of $M$.
We apply all these notions to unbased spaces and maps by adjoining disjoint basepoints.
These definitions are clearly homotopy invariant, and are known to agree with all other definitions of Lefschetz number and Euler characteristic; see~\cite{dp:duality}.

There are many reasons why this formulation of Lefschetz number and Euler characteristic is useful, but for us the most important is that it makes it easy to prove the multiplicativity theorem for trivial bundles.
Specifically, if $M$ and $N$ are both $n$-dualizable based spaces, then it is easy and formal to prove that $M\sm N$ is also $n$-dualizable; its dual is $\dual{N}\sm \dual{M}$.
Moreover, if $f\colon M\to M$ and $g\colon N\to N$ are endomorphisms, we can prove by formal manipulation that $\tr(g \sm f) \sim \tr(g) \circ \tr(f)$ as maps $S^n \to S^n$, hence $L(g \sm f) = L(g) \cdot L(f)$.
Of course, if $M$, $N$, $f$, and $g$ are unbased, then $M_+ \sm N_+ \cong (M\times N)_+$ and $f_+ \sm g_+ \cong (f\times g)_+$, so we obtain the multiplicativity theorem for trivial bundles.

\section{Linearity of the Lefschetz number}
\label{sec:lefschetz}

We now generalize this by considering \emph{diagrams} of spaces.
Let $A$ and $B$ be small categories; by a \textbf{profunctor} from $A$ to $B$, we will mean a functor from $A\times B\op$ to the category \ptop of based spaces.
For example, for any small category $A$ and any based space $X$, we have a profunctor $X \sm A$ from $A$ to $A$, defined by
\[ (X\sm A)(a,a') = X\sm \hom_A(a',a)_+, \]
where $\hom_A(a',a)$ has the discrete topology.
More generally, if $M$ is a profunctor from $A$ to $B$ and $X$ is a based space, we have a profunctor $X\sm M$ defined by $(X\sm M)(a,b) = X\sm M(a,b)$.

If $M\colon A\times B\op \to \ptop$ and $N\colon B\times C\op\to\ptop$ are profunctors, we define their \textbf{composite} $M\sm_B N$ to be the profunctor from $A$ to $C$ where $(M\sm_B N)(a,c)$ is the \emph{homotopy coend} of the $(B\op\times B)$-indexed diagram $(b,b') \mapsto M(a,b) \sm N(b',c)$.
In good situations (such as when $M$ and $N$ are nondegenerately based and have the homotopy type of CW complexes), this can be constructed as the geometric realization of its \emph{simplicial bar construction}, which is a simplicial space:
\[ \xymatrix{
  \cdots \ar@<2mm>[r] \ar@{<-}@<1mm>[r] \ar[r] \ar@{<-}@<-1mm>[r] \ar@<-2mm>[r] &
  \coprod_{b,b'} M(a,b) \sm \hom_B(b',b)_+ \sm N(b',c)  \ar@<2mm>[r] \ar@{<-}[r] \ar@<-2mm>[r] &
  \coprod_{b} M(a,b) \sm N(b,c)}
  \]
See e.g.~\cite{may:csf,meyer:bar_i,shulman:htpylim} for more details.
We have the following natural equivalences for any profunctor $M\colon A\times B\op \to \ptop$ and any based space $X$.
\[(X\sm A) \sm_A M \simeq X \sm M \qquad
M \sm_B (X\sm B) \simeq X\sm M
\]

We are particularly interested in profunctors either to or from the terminal category $\tc$ (with one object and only its identity morphism).
A profunctor $M$ from $A$ to $\tc$ is the same as an $A$-shaped diagram in \ptop, while a profunctor $\Phi$ from $\tc$ to $A$ is the same as an $A\op$-shaped diagram in $\ptop$.
Their composite $\Phi\sm_A M$ is called the \textbf{$\Phi$-weighted homotopy colimit of $M$} and denoted $\colim^\Phi(M)$.
In particular, if $\Phi$ is constant at $S^0$, then $\colim^\Phi(M)$ is homotopy equivalent to the ordinary homotopy colimit of $M$.
The general framework for linearity will tell us how to compute traces of endomaps of $\Phi\sm_A M$ that are induced by endo-natural-transformations of $M$.

In order to state such a result, we need an appropriate notion of duality.
We would like to say that a profunctor $M$ from $A$ to $B$ is \textbf{right} $n$-\textbf{dualizable} if there is a profunctor $\rdual M$ from $B$ to $A$ and natural transformations
$\eta \colon S^n\sm A\to M\sm_B \rdual{M}$ and $\epsilon\colon \rdual{M}\sm_A M \to S^n\sm B$ such that the composites
\[\xymatrix@R=14pt{S^n\sm M\ar[d]^-\wr&& S^n\sm M\\
 (S^n\sm A)\sm_A M\ar[r]^-{\eta\sm \id}& M\sm_B \rdual{M}\sm_A M\ar[r]^-{\id\sm \epsilon} &M\sm_B (S^n\sm B)\ar[u]^-\wr\\
 \rdual M\sm_A (S^n\sm A)\ar[r]^-{\id \sm  \eta}& \rdual M\sm_A M\sm_B \rdual M
\ar[r]^-{\epsilon\sm \id}&(S^n\sm B)\sm_B \rdual{M} \ar[d]_-\wr\\
S^n\sm\rdual M \ar[u]_-\wr&&S^n\sm \rdual M
}\]
become naturally homotopic to identity maps\footnote{In \S\ref{sec:lefschetz-number} we said instead that the corresponding composites should be homotopic to transpositions such as $S^n\sm M \simeq M\sm S^n$.  We have now incorporated these transpositions into the equivalences at the beginning and end of the displayed composites.} after smashing with some $S^m$.

However, here there is an additional wrinkle: we need to allow maps $\eta$ and $\epsilon$ where the domains have been modified.   Formally, we allow them to be
\emph{cofibrantly replaced} with respect to the \emph{projective model structure}  \cite[\S11.6]{hirschhorn}.
We think of this as replacing the given diagram 
by a \emph{cell complex} whose cells are \emph{freely generated} at some object of the indexing category.  We will not need to be more specific about this here, since in many cases of interest (see \autoref{eg:smcofibers} below)
a very simple modification suffices.
In the more general cases, \cite{PS5} allows us to avoid the issue by using derivators, which ``package'' the fibrant and cofibrant replacements into an abstract structure so that we don't have to think about them.

The asymmetry between the categories $A$ and $B$ in the above definition means that it incorporates two very different types of duality, which can be best understood by restricting to the case when either $A$ or $B$ is the terminal category $\tc$.
On the one hand, if $M$ is an $A$-shaped diagram, we can think of it as a profunctor from $A$ to $\tc$.
If the resulting profunctor is right $n$-dualizable, we say that $M$ is \textbf{pointwise dualizable}.
On the other hand, if $\Phi$ is a $B\op$-shaped diagram, we can regard it as a profunctor from $\tc$ to $B$.
If this profunctor is right $n$-dualizable, we say that $\Phi$ is \textbf{absolute}.

The reason for the name ``pointwise dualizable'' is the following lemma.

\begin{lem}[{\cite[Lemma 3.5]{PS5}}]\label{thm:smpointwise-dual}
%[{\cite[\autoref{PS5:thm:smcpwdual}]{PS5}}]\label{thm:smpointwise-dual}
  A functor $M\colon A \to \ptop$ is pointwise dualizable if and only if each space $M(a)$ is $n$-dualizable in $\ptop$.
\end{lem}

Absoluteness is a much stronger condition than pointwise dualizability.
However, the \emph{stability} built into the definition of $n$-duality allows us to exhibit a few examples.

\begin{eg}\label{eg:sminitial}
  Let $B=\emptyset$ be the empty category, and $\Phi\colon \emptyset\op\to\ptop$ the unique functor, regarded as a profunctor from $\tc$ to $\emptyset$.
  Then we can take $\rdual{\Phi}$ to also be the unique functor $\emptyset\to \ptop$, now regarded as a profunctor from $\emptyset$ to $\tc$.
  The composite $\rdual{\Phi}\sm_{\tc} \Phi$ is again the unique functor $\emptyset\times\emptyset\op \to \ptop$, while the composite $\Phi \sm_\emptyset \rdual{\Phi}$ is the functor $\tc \times \tc\op \to \ptop$ sending the unique object to $\pt$ (the one-point based space).
  The maps $\eta \colon S^0 \sm \tc \to \Phi \sm_\emptyset \rdual{\Phi}$ and $\epsilon \colon \rdual{\Phi}\sm_{\tc} \Phi \to S^0 \sm \emptyset$ are the only possible ones (no cofibrant replacement is necessary), and the required identities hold vacuously.
  Thus, $\Phi$ is absolute.
\end{eg}

\begin{eg}\label{eg:smcoprod}
  Let $B = \twc$ be the category with two objects and only identity morphisms, and $\Phi\colon (\twc)\op \to \ptop$ the functor sending each object to $S^0$, regarded as a profunctor from $\tc$ to $\twc$.
  Then we can take $\rdual{\Phi}$ to be the functor $\twc\to\ptop$ sending each object to $S^1$, regarded as a profunctor from $\twc$ to $\tc$.
  The composite $\rdual{\Phi} \sm_{\tc} \Phi$ is the functor $(\twc) \times (\twc)\op\to\ptop$ sending each object to $S^1$, while the composite $\Phi \sm_B \rdual{\Phi}$ is the functor $\tc \times \tc\op\to \ptop$ sending the unique object to $S^1 \vee S^1$ (the coproduct in \ptop of two copies of $S^1$).
  The coevaluation $\eta \colon S^1 \sm \tc \to \Phi \sm_B \rdual{\Phi}$ is the ``pinch'' map that wraps $S^1$ once around each summand of $S^1 \vee S^1$, while the evaluation $\epsilon \colon \rdual{\Phi} \sm_{\tc} \Phi \to S^1 \sm (\twc)$ is the identity on $S^1$ where possible, and otherwise the zero map.
  Again, no cofibrant replacement is necessary, and the required identities are easy to check.
  Thus, $\Phi$ is absolute.
\end{eg}

The interaction of pointwise duality and absoluteness is central to the subject.
For instance, we have the following theorem.

\begin{thm}\label{thm:smcompose-dual}
  If $\Phi$ is absolute and $M$ is pointwise dualizable, then the $\Phi$-weighted homotopy colimit $\colim^\Phi(M) = \Phi\sm_A M$ is $n$-dualizable.
  (More generally, if $X$ and $Y$ are right $n$-dualizable profunctors from $A$ to $B$ and from $B$ to $C$ respectively, then their composite $X\sm_B Y$ is also right $n$-dualizable.)
\end{thm}

Perhaps surprisingly, the proof of this theorem is quite trivial: the dual of $X\sm_B Y$ is $\rdual{Y} \sm_B \rdual{X}$ and the rest is diagram chases.

Note that the $\Phi$ of Examples \ref{eg:sminitial} and \ref{eg:smcoprod} are of the ``constant at $S^0$'' sort, so that in both cases $\Phi$-weighted homotopy colimits can be identified with ordinary homotopy colimits.
Thus, in both cases, homotopy colimits of pointwise dualizable diagrams are dualizable.

\begin{eg}
  If $B=\emptyset$, there is only one diagram $B\to\ptop$, which is trivially pointwise dualizable.
  Thus, \autoref{thm:smcompose-dual} implies that its (homotopy) colimit, which is the one-point based space $\pt$, is $n$-dualizable.
\end{eg}

\begin{eg}
  If $B=\twc$, then a diagram $B\to\ptop$ is just a pair of based spaces $X$ and $Y$, and by \autoref{thm:smpointwise-dual} it is pointwise dualizable just when $X$ and $Y$ are $n$-dualizable.
  Thus, \autoref{thm:smcompose-dual} implies that the homotopy colimit of such a diagram, which (at least if $X$ and $Y$ are nondegenerately based) is just the wedge $X\vee Y$, is also $n$-dualizable.
\end{eg}

These conclusions (the initial object is dualizable, and coproducts of dualizable spaces are dualizable) are easy to prove by other means.
Somewhat less obvious (though also well-known) is the following example, in which the weight is non-constant.

\begin{eg}\label{eg:smcofibers}
  Let $B$ be the arrow category, with two objects $a$ and $b$, and one nonidentity morphism $\alpha \colon a\to b$.
  Let $\Phi \colon B\op\to \ptop$ be defined by $\Phi(b) = S^0$ and $\Phi(a) = \pt$, with $\Phi(\alpha)\colon S^0 \to \pt$ the unique map.
  Then we can take $\rdual{\Phi}\colon B\to \ptop$ to have $\rdual{\Phi}(a) = S^0$ and $\rdual{\Phi}(b) = \pt$, with $\rdual{\Phi}(\alpha)\colon S^0 \to \pt$ again the unique map.
  
  Their composite $\Phi \sm_{B} \rdual{\Phi}$ is the homotopy pushout of the diagram 
  \begin{equation}
  \vcenter{\xymatrix{
      S^0 \sm S^0\ar[r]\ar[d] &
      S^0 \sm \pt\\
      \pt \sm S^1&
    }}
  \qquad\text{or equivalently}\qquad
  \vcenter{\xymatrix{
      S^0\ar[r]\ar[d] &
      \pt\\
      \pt&
    }}
  \end{equation}
  which is the suspension of $S^0$, namely $S^1$.
  Technically by this we mean the profunctor $S^1\sm \tc$; thus we can take $n=1$ and let the coevaluation $\eta \colon  S^1 \sm \tc \to \Phi \sm_{B} \rdual{\Phi}$ be the isomorphism.

  Their composite $\rdual{\Phi} \sm_{\tc} \Phi$ in the other order is the $(B\times B\op)$-indexed diagram that is $\pt$ everywhere except that $(\rdual{\Phi} \sm_{\tc} \Phi)(a,b) = S^0$, which can be drawn like the square on the left below (with the $B$ direction horizontal and the $B\op$ direction vertical).
  \begin{equation}
    \vcenter{\xymatrix@-.5pc{
        \pt\ar[r]\ar@{<-}[d] &
        \pt\ar@{<-}[d]\\
        S^0\ar[r] &
        \pt
      }}
    \hspace{2cm}
    \vcenter{\xymatrix@-.5pc{
        S^1\ar[r]\ar@{<-}[d] &
        S^1\ar@{<-}[d]\\
        \pt\ar[r] &
        S^1
      }}
  \end{equation}
  Similarly, the $(B\times B\op)$-indexed diagram $S^1\sm B$ can be drawn like the square on the right above.
  Here we do need a cofibrant replacement in order to have a nontrivial map $\rdual{\Phi} \sm_{\tc} \Phi \to S^1\sm B$, and as remarked previously, in this case we can use a \emph{Reedy model structure} for the diagram shape $B\times B\op$.
  The resulting cofibrant replacement of $\rdual{\Phi} \sm_{\tc} \Phi$ looks like this:
  \begin{equation}
    \vcenter{\xymatrix@-.5pc{
        D^1\ar[r]\ar@{<-}[d] &
        D^2\ar@{<-}[d]\\
        S^0\ar[r] &
        D^1
      }}
  \end{equation}
  and the evaluation $\epsilon\colon \rdual{\Phi} \sm_{\tc} \Phi \to S^1\sm B$ wraps each $D^1$ around the corresponding $S^1$.
  The required identities hold, and so $\Phi$ is absolute.

  Now a diagram $B\to\ptop$ is just a pair of based spaces and a map between them, $i\colon Y\to X$, and by \autoref{thm:smpointwise-dual} it is pointwise dualizable just when $X$ and $Y$ are $n$-dualizable.
  The $\Phi$-weighted colimit of such a diagram is the homotopy pushout of
  \begin{equation}
    \vcenter{\xymatrix{
        S^0 \sm Y\ar[r]\ar[d] &
        S^0 \sm X\\
        \pt \sm Y &
      }}
    \qquad\text{or equivalently}\qquad
    \vcenter{\xymatrix{
        Y\ar[r]\ar[d] &
        X\\
        \pt &
      }}
  \end{equation}
  which is to say the \emph{mapping cone} of $i$.
  Thus, \autoref{thm:smcompose-dual} implies that the cone of any map between $n$-dualizable spaces is $n$-dualizable.
\end{eg}

Next we define traces for endomorphisms of dualizable profunctors.
However, there is a problem: the isomorphism $M\sm \rdual{M} \cong \rdual{M} \sm M$ in~\eqref{eq:tr} doesn't even make sense for profunctors, since $M\sm_B \rdual{M}$ and $\rdual{M} \sm_A M$ are not even diagrams of the same shape.
For this purpose we introduce a ``quotienting'' operation that makes them both into ordinary spaces, following~\cite{kate:traces,PS2}.

Specifically, if $M$ is any profunctor from $A$ to $A$, then we define its \textbf{shadow}, denoted $\sh{M}_A$, to be its homotopy coend.
If $M$ and $N$ are profunctors from $A$ to $B$ and from $B$ to $A$ respectively, then we have a natural equivalence
\[\sh{M\sm_B N}_A \simeq \sh{N\sm_A M}_B.\]
If $M = S^n \sm A$, then its shadow is the $n$-fold suspension of a space representing the ``Hochschild homology'' of the category $A$.
We write $\sh{A}_A$ for $\sh{S^0 \sm A}_A$.

The stable maps $S^n\to \sh{S^n\sm A}_A$ play an essential role in this section, so it will be necessary to have a more explicit understanding of them.
First observe that smashing with $S^n$ commutes with the homotopy coend, so it is enough to consider stable maps $S^0\to \sh{A}_A$, i.e.\ elements of the $0^{\mathrm{th}}$ stable homotopy group $\pi_0^s(\sh{A}_A)$ of $\sh{A}_A$.
Now for an unbased space $Y$, there is an integer $n$ so that $\pi_0^s(Y_+)\cong \pi_n(\Sigma^n(Y_+))$.
Since $\tilde{H}_n(\Sigma^n (Y_+))\cong  H_0(Y)$, the Hurewicz theorem implies that $\pi_0^s(Y_+)\cong H_0(Y)\cong \mathbb{Z}\pi_0(Y)$ is the free abelian group on the set of connected components of $Y$.
Therefore, to compute $\pi_0^s(\sh{A}_A)$, it will suffice to show that $\sh{A}_A$ has a disjoint basepoint and to count the other components.

For this purpose, recall that $(S^0 \sm A)(a,a') = \hom_A(a',a)_+$.
Homotopy coends commute with adjoining disjoint basepoints, so the homotopy coend of $S^0 \sm A$ is the homotopy coend of the diagram $\hom_A$ of unbased spaces, with a disjoint basepoint.
This homotopy coend, in turn, is the geometric realization of the simplicial set
\[ \xymatrix{
  \cdots \ar@<2mm>[r] \ar@{<-}@<1mm>[r] \ar[r] \ar@{<-}@<-1mm>[r] \ar@<-2mm>[r] &
  \coprod_{a,a'} \hom_A(a,a') \times \hom_A(a',a) \ar@<2mm>[r] \ar@{<-}[r] \ar@<-2mm>[r] &
  \coprod_{a} \hom_A(a,a)}
  \]
Therefore, its set of connected components is the set of endomorphisms $\alpha\colon a\to a$ in $A$, modulo the equivalence relation generated by $\alpha \beta \sim \beta\alpha$ whenever $\alpha\colon a\to b$ and $\beta\colon b\to a$ are morphisms that are composable in both orders.
We call these \textbf{conjugacy classes} of $A$, since when $A$ is a group $G$ regarded as a one-object groupoid, they are precisely the conjugacy classes of $G$.
Thus, $\pi_0^s(\sh{A}_A)$ is the free abelian group on the conjugacy classes of $A$.
We write the conjugacy class of $\alpha\colon a\to a$ as $[\alpha]$.

We can now define traces for dualizable profunctors.
If $M$ is a right dualizable profunctor from $A$ to $B$, the \textbf{trace} of a map $f\colon M \to M$ is the composite
\[\xymatrix{\sh{S^n\sm A}_A\ar[d]^-\eta&&\sh{S^n\sm B}_B\\
\sh{M\sm_B  \rdual M}_A\ar[r]^{f\sm \id}&
\sh{M\sm _B \rdual M}_A\ar[r]^\cong&\sh{\rdual M\sm_A M}_B\ar[u]^-{\epsilon}\\
}\]

The following is the central theorem of the subject: the \emph{origin of all linearity formulas}.
It says that in the situation of \autoref{thm:smcompose-dual}, we can calculate the trace of an induced map on the colimit from the trace of an endomap of a diagram.
Like \autoref{thm:smcompose-dual}, its proof is essentially entirely formal manipulation.

\begin{thm}[{\cite[Prop.~7.5]{PS2}}]\label{thm:dercomposite}
  If $M\colon A\to\ptop$ is pointwise dualizable and $\Phi\colon A\op\to\ptop$ is absolute, then for any $f\colon M\to M$ and $g\colon \Phi\to\Phi$ we have
  \[ \tr(g\sm_A f) = \tr(f) \circ \tr(g). \]
  (More generally, if $X$ and $Y$ are right $n$-dualizable profunctors from $A$ to $B$ and from $B$ to $C$ respectively, then for any $f\colon Y\to Y$ and $g\colon X\to X$ we have $\tr(g\sm_B f) = \tr(f) \circ \tr(g)$.)
\end{thm}

In our applications, we will often take $g=\id_\Phi$ (although in the proof of \autoref{thm:smpointwise-trace} we will need the more general version).
In this case, $\id_\Phi \sm_A f$ is just the endomorphism of $\colim^\Phi(M)$ induced by $f$, which we denote $\colim^\Phi(f)$; thus we have
\begin{equation}
  \tr(\colim^\Phi(f)) = \tr(f) \circ \tr(\id_\Phi).\label{eq:dercomposite}
\end{equation}
In order to make practical use of this, we need to be able to identify $\tr(f)$ and $\tr(\id_\Phi)$.
Note that $\tr(f)$ is a stable map $\sh{A}_A \to S^0$, i.e.\ an element of the stable cohomotopy of $\sh{A}_A$.
This could be hard to compute completely.
However, for~\eqref{eq:dercomposite} all that matters is its composite with the stable map $\tr(\id_\Phi)\colon S^0 \to \sh{A}_A$, which is an element of $\pi^s_0(\sh{A}_A)$.
Thus, it will be enough to know the composites of $\tr(f)$ with the generators of $\pi^s_0(\sh{A}_A)$, which as we have seen are induced by conjugacy classes in $A$.
This is what the following lemma provides.

\begin{lem}[{\cite[Lemma 5.11]{PS5}}]\label{thm:smpointwise-trace}
	%[{\cite[\autoref{PS5:thm:smderomega}]{PS5}}]\label{thm:smpointwise-trace}
  If $M\colon A\to\ptop$ is pointwise dualizable and $f\colon M\to M$, then for any conjugacy class $[a\xto{\alpha} a]$ in $A$, the composite
  \begin{equation}
    \xymatrix{ S^n \ar[r]^-{[\alpha]} & S^n \sm \sh{A}_A \ar[r]^-{\tr(f)} & S^n }\label{eq:smderomtr}
  \end{equation}
  is stably equal to the trace (in \ptop) of the composite
  \begin{equation}
    \xymatrix{ M_a \ar[r]^-{M_\alpha} & M_a \ar[r]^-{f_a} & M_a }\label{eq:smderomres}
  \end{equation}
   where $M_a$ is the value of $M$ on $a$.
\end{lem}
\begin{proof}[Sketch of proof]
  Define a profunctor $\Phi_a$ from $\tc$ to $A$ by
  \[ \Phi_a(a') = \hom_A(a',a)_+. \]
  This is absolute; its right dual is defined by
  \[ (\rdual{\Phi_a})(a') = \hom_A(a,a')_+. \]
  The evaluation and coevaluation are defined by the inclusion of the identity map and the composition of morphisms.
  Moreover, the $\Phi_a$-weighted colimit of $M$ is just the object $M_a$.
  We also have an endomorphism $\omega_\alpha\colon \Phi_a \to \Phi_a$ induced by composition with $\alpha$, and $\omega_\alpha \sm_A f$ is the composite $f_a \circ M_\alpha$ appearing in \eqref{eq:smderomres}.
  Finally, $\tr(\omega_\alpha)$ can be shown to be $[\alpha]\colon S^n \to S^n \sm \sh{A}_A$, so \autoref{thm:dercomposite} implies the desired statement.
\end{proof}

Thus, in order to apply \autoref{thm:dercomposite} for any particular weight $\Phi$, it suffices to show that $\Phi$ is absolute, and calculate $\tr(\id_\Phi)$.
We refer to $\tr(\id_\Phi)$ as the \textbf{coefficient vector} of $\Phi$.
Recall that it is an element of the 0th stable homotopy group of $\sh{A}_A$, hence a linear combination $\sum_{[\alpha]} \phi_{[\alpha]} \cdot [\alpha]$ of the generators $[\alpha]$ corresponding to conjugacy classes in $A$, with integer coefficients $\phi_{[\alpha]}$.
Combining this with \autoref{thm:smpointwise-trace}, we obtain:

\begin{thm}\label{thm:lin}
  If $M\colon A\to\ptop$ is pointwise dualizable, $\Phi\colon A\op\to\ptop$ is absolute, and $f\colon M\to M$, then we have
  \[ \tr(\colim f) = \sum_{[\alpha]} \phi_{[\alpha]} \tr(f_a \circ M_\alpha) \]
  where $\phi_{[\alpha]}$ are the coefficients of $\Phi$.
\end{thm}

\begin{eg}
  If $A$ is the empty category there is a unique functor $M\colon A\to\ptop$, which is trivially pointwise dualizable; hence its colimit, which is the one-point space $\pt$, is dualizable.
 The shadow of $A$ is also $\pt$, so the trace of the unique endomorphism of $\pt$ is the composite $S^0 \to \pt \to S^0$, i.e.\ the zero endomorphism of $S^0$.
\end{eg}

\begin{eg}\label{eg:lefschetz-coproducts}
  If $A=\twc$ then 
  we have $\pi_0^s(\sh{A}) \cong \mathbb{Z}\oplus\mathbb{Z}$.  For a pointwise dualizable $M\colon \twc\to \ptop$ and $f\colon M\to M$, \autoref{thm:smpointwise-trace} implies that 
  the map $\pi_0^s(\sh{A}_A)\to \pi_0^s(S^0)$ induced by $\tr(f)$ is the row vector composed of $\tr(f_a)$ and $\tr(f_b)$.
  Thus, we have
  \[\tr(f_a \vee f_b) = \phi_a \cdot \tr(f_a) + \phi_b\cdot \tr(f_b)\]
  for some endomorphisms $\phi_a, \phi_b$ of $S^n$.
  Knowing that such $\phi_a$ and $\phi_b$ exist, and are the same for all $M$ and $f$, enables us to calculate them easily.
  Namely, let $M_a = \pt$ and $M_b=S^0$ and let $f$ be the identity.
  Then $\tr(f_a)=0$ by the previous example, and $\tr(f_b)=1$ since it is the identity; while $M_a \vee M_b \cong S^0$ and $f_a \vee f_b = \id_{S^0}$, so that $\tr(\colim^\Phi f)=1$ as well.
  Thus, $1 = \phi_a \cdot 0 + \phi_b \cdot 1$, so $\phi_b = 1$.
  Similarly, $\phi_a=1$, so our linearity formula is
  \[\tr(f_a \vee f_b) = \tr(f_a) + \tr(f_b).\]
  This is, of course, well-known and easy to prove by other means.
\end{eg}

Finally, applying this to \autoref{eg:smcofibers}, we obtain the linearity of Lefschetz numbers.

\begin{thm}\label{Lefschetz}
  Suppose $X$ and $Y$ are closed smooth manifolds, and we have continuous maps $i\colon Y\to X$, $f\colon X\to X$ and $g\colon Y\to Y$ so that 
  $f\circ i=i\circ g$.  Then  \[L(f)-L(g)=L(h)\] where $h$ is the endomorphism of the mapping cone of $i$ induced by $f$ and $g$.
\end{thm}

 As a special case of this theorem, if $A\subset X$ are closed smooth manifolds and $f\colon X\to X$ is a continuous map so that $f(A)\subset A$ then 
\[L(f)-L(f|A)=L(f/A)\] since the mapping cone on the inclusion is equivalent to the quotient space $X/A$.
\begin{proof}  Let $B$ be the the category with two objects $a$ and $b$ and one nonidentity morphism $\alpha\colon a \to b$.  
  In \autoref{eg:smcofibers}, we saw that the functor $\Phi\colon B\op\to \ptop$ defined by $\Phi(b)=S^0$ and $\Phi(a)=\pt$ is absolute, and that $\Phi$-weighted colimits are mapping cones.

  Now $B$ has two conjugacy classes, $[\id_a]$ and $[\id_b]$, so $\pi_0^s\left(\sh{B}_B\right) \cong \mathbb{Z}\oplus \mathbb{Z}$.
  Thus, $\tr(\id_\Phi)$ is determined by two integers $\phi_a$ and $\phi_b$, and for any pointwise dualizable profunctor $M\colon B\to \ptop$ and 
  endomorphism $f\colon M\to M$ we have a linearity formula
  \[ \tr(\colim^\Phi (f)) = \phi_a \cdot \tr(f_a) + \phi_b \cdot \tr(f_b).\]
  As in \autoref{eg:lefschetz-coproducts}, now that we know that $\phi_a$ and $\phi_b$ exist, we can calculate them by considering some very special cases.
  On one hand, if $M$ is the diagram  $(S^n\to S^n)$, then its cofiber is $\pt$.  If $f$ is the identity, we have
  \[ 0 = \phi_a \cdot 1 + \phi_b \cdot 1.\]
  On the other hand, if $M$ is the diagram $(\pt\to S^n)$, then its cofiber is $S^n$.  If $f$ is again the identity, we have
  \[ 1 = \phi_a \cdot 0 + \phi_b \cdot 1.\]
  Solving these equations, we obtain $\phi_b = 1$ and $\phi_a = -1$.

  Thus, \autoref{thm:lin} tells us that if $M\colon B\to \ptop$ is pointwise dualizable and $f\colon M\to M$, then
  \[ \tr(\colim(f)) = \tr(f_b) - \tr(f_a). \]
  If we let $M_a=Y$, $M_b=X$, $f_a=g$ and $f_b=f$ then we recover the statement of the theorem.
\end{proof}

It is important to note that this approach is not tied to this particular diagram shape.  The same proof can be applied to 
any small category $B$ and functor $\Phi\colon B\op\to \ptop$ 
\emph{as long as we can show that the  functor $\Phi$ is 
absolute}.   This is a significant restriction, since verifying absoluteness can be difficult, but it also does allow
for significant generalizations; a number of examples can be found in \cite{PS5}.

The approach is also not tied to the particular context of ordinary spaces: it works in any ``homotopy theory''.
In~\cite{PS5} we describe it in the general context of \emph{derivators}; the proof of \autoref{Lefschetz} then applies in any \emph{stable} derivator, i.e.\ one where the suspension functor is an equivalence.
(To put based spaces in such a context, one works instead with spectra.)
A similar general form of \autoref{Lefschetz} was proven in~\cite{add} using triangulated categories and stable model categories.
By~\cite{groth:ptstab}, any stable derivator has an underlying triangulated category, and any stable model category gives rise to a stable derivator; thus our general form of \autoref{Lefschetz} has roughly the same generality as that of~\cite{add}.
However, as remarked above, our method also applies to more general diagram shapes.
And, as we now proceed to explain, it also applies to the Reidemeister trace.

\section{Parametrized spaces and the Reidemeister trace}
\label{sec:fibration}

We now introduce the context of parametrized spaces and ex-fibrations following~\cite{maysig:pht}, which will enable us to view the Reidemeister trace as a trace in a way similar to
the perspective on the Lefschetz number described above.

For us, \emph{fibration} will always mean \emph{Hurewicz fibration}.
If $p_1\colon E_1 \to X$ and $p_2\colon E_2 \to X$ are fibrations, a \textbf{fiberwise map} is a map $f\colon E_1\to E_2$ such that $p_2 f = p_1$.
Similarly, a \textbf{fiberwise homotopy} is a map $H\colon E_1\times [0,1]\to E_2$ such that $p_2(H(e,t)) = p_1(e)$ for all $e\in E$ and $t\in [0,1]$.
We denote fiberwise homotopy equivalences by $\simeq$.

If, on the other hand, we have fibrations $p_1\colon E_1 \to X_1$ and $p_2\colon E_2 \to X_2$ over (possibly) different base spaces, and $\fbar\colon X_1 \to X_2$ is a map, then by a 
\textbf{(fiberwise) map over \fbar} we mean a map $f\colon E_1\to E_2$ such that $p_2 f = \fbar p_1$.
Thus, a ``fiberwise map'' without qualification is always over the identity $\id_X$.

Two basic constructions on fibrations are \emph{pullback} and \emph{pushforward} along a continuous map $\fn\colon Y\to X$.
Firstly, given a fibration $p\colon E\to X$, we have a pullback fibration $\fn^*E \to Y$.
For any other fibration $E'\to Y$, there is a natural bijection between maps $E'\to E$ over $\fn$ and fiberwise maps $E'\to \fn^*E$ (over $\id_Y$).
Secondly, given a fibration $q\colon E\to Y$, the composite $E\xto{q} Y \xto{\fn} X$ is not in general a fibration (though it is if $\fn$ is also a fibration), but up to homotopy we can replace it by one.
We denote the result by $\fn_! E \to X$.

A particularly important construction built out of these operations is the following.
Given fibrations $M \to Y\times X$ and $N\to X\times Z$, we can form the product fibration $M\times N \to Y\times X\times X\times Z$, pull back along the diagonal $Y\times X\times Z \to Y\times X\times X\times Z$, then compose with the projection $Y\times X\times Z \to Y\times Z$ (which is a fibration).
This yields a fibration $M\times_{X} N \to Y\times Z$ whose total space is the pullback of $M$ and $N$ over $X$.
(We could have defined this more directly, but the description given above generalizes more easily.)

The fibrational version of a based space is a \emph{sectioned} fibration, i.e.\ a fibration $p\colon E\to X$ equipped with a continuous section $s\colon X\to E$ (so that $p s = \id_X$).
We usually assume that the section is a fiberwise closed cofibration, in which case we call $p$ an \textbf{ex-fibration} (see~\cite{maysig:pht} for details).
We say that a fiberwise map 
between ex-fibrations is an \textbf{ex-map} if it also commutes with the specified sections.
Note that if $p\colon E\to X$ is an ex-fibration, then each fiber $\fib{p}{b}$ is a nondegenerately based space.

If $E\to X$ is any fibration, then $E\sqcup X \to X$ is an ex-fibration, called the result of \emph{adjoining a disjoint section} to $E$;
we denote it by $E_{+X}$.
Additionally, for any ex-fibration $E\to X$ and any (nondegenerately) based space $W$, we have an induced ex-fibration $W\sm E \to X$, whose fiber over $b\in X$ is the ordinary smash product $W \sm \fib{p}{b}$.
Ex-fibrations can be pulled back along a continuous map $\fn\colon Y\to X$ in a straightforward way.
They can also be pushed forwards, by composing and then pushing out along the section (then replacing by a fibration, if necessary).

Two ex-fibrations $p\colon M\to Y$ and $q\colon N\to X$ have an \textbf{external smash product} $M\exsm N \to Y\times X$, whose fiber over $(a,b)$ is $\fib{p}{a}\sm \fib{q}{b}$.
If we have 
 ex-fibrations $p\colon M\to Y\times X$ and $q\colon N\to X\times Z$ the  \textbf{smash pullback} of $M$ and $N$ is the result of pulling back $M\exsm N$ along $Y\times X\times Z \to Y\times X\times X\times Z$, then pushing forward to $Y\times Z$.
This yields an ex-fibration $M\odot N \to Y\times Z$; its fiber over $(a,c)$ consists of all the smash products $\fib{p}{a,b} \sm \fib{q}{b,c}$ for all $b\in X$, with their basepoints identified (and with a suitable quotient topology which relates these products for different $b$).
The smash pullback is associative.

\begin{rmk}\label{rmk:ordering-convention}
  If we are given a fibration $M\to X$, we may regard it either as a fibration $M\to X\times \pt$ or as a fibration $M\to \pt\times X$.
  We will only need the second and we will denote it by $\widecheck{M}$.

  Because we will have another use for subscripts on $\odot$ in \S\ref{sec:parametrize}, we will not indicate in the notation the space that is ``canceled'' in the smash pullback.
  Instead we will rely on the understanding that $M\odot N$ is a space over the left-hand space of $M$ and the right-hand space of $N$, as above.
\end{rmk}

For any space $X$, let $P X$ denote the space of paths $\gamma \colon [0,1]\to X$, and define the fibration $P X \to X\times X$ by evaluation at the endpoints, $\gamma \mapsto (\gamma(0),\gamma(1))$.
Then  $(P X)_{+X\times X}$ is a  two-sided homotopy unit for the smash pullback; we denote these units by $U_X \to X\times X$.
For any (nondegenerately) based space $W$ and ex-fibration $p\colon M\to Y\times X$ we have natural equivalences 
\[(W\wedge U_Y)\odot M\simeq W\wedge M \qquad M\odot (W\wedge U_X)\simeq W\wedge M\] 

Given a map $\fn\colon Y\to X$ of base spaces, we write $X_\fn$ for the 
ex-fibration $((\id\times \fn)^* P X)_{+(X\times Y)}$ and ${}_\fn X$ for the ex fibration $((\fn\times \id)^* P X )_{+(Y\times X)}$.
The ex-fibrations $X_\fn$ and ${}_\fn X$ are called \textbf{base change objects}, and have the following important properties:
\begin{itemize}
\item For any ex-fibration $M\to Z\times X$, we have $M\odot X_\fn \simeq (\id\times \fn)^* M$.
\item For any ex-fibration $M\to X\times Z$, we have ${}_\fn X \odot M \simeq (\fn\times \id)^* M$.
\item For any ex-fibration $M\to Z\times Y$, we have $M\odot {}_\fn X \simeq (\id\times \fn)_! M$.
\item For any ex-fibration $M\to Y\times Z$, we have $X_\fn \odot M \simeq (\fn\times \id)_! M$.
\end{itemize}

If $M\to Y\times Y$ is an ex-fibration, we denote by $\sh{M}$ its \textbf{shadow}, which is the based space defined by pulling back along the diagonal $Y\to Y\times Y$, then quotienting out the section (i.e.\ pushing forward along the unique map $Y\to \pt$).
For ex-fibrations $M \to Y\times X$ and $N\to X\times Y$, we have a canonical equivalence
\[ \sh{M \odot N} \;\simeq\; \sh{N\odot M}. \]
For a based space $X$ and an ex-fibration $M\to Y\times Y$, we have $\sh{X\sm M} \simeq X \sm\; \sh{M}$.

The shadow of the unit ex-fibration $U_Y$ is $(\Lambda Y) _+$, the free loop space of $Y$ with a disjoint basepoint.
Similarly, for an endomorphism $\fn\colon Y\to Y$, the shadow of the base change object $Y_\fn$ is $(\Lambda^\fn Y)_+$, where $\Lambda^\fn Y$ denotes the \textbf{twisted free loop space}: its points are pairs consisting of a point $a\in Y$ and a path $a\leadsto \fn(a)$.
Likewise, the shadow of ${}_{\fn} Y$ is the space of paths $\fn(a)\leadsto a$ (with a disjoint basepoint), which is homeomorphic to $(\Lambda^\fn Y)_+ \cong \sh{Y_\fn}$.

We say that an ex-fibration $M\to Y\times X$ is \textbf{(right) $n$-dualizable} if there is an ex-fibration $\rdual{M}\to X\times Y$ and fiberwise ex-maps
\begin{align*}
  \eta &\colon S^n \sm U_Y \too M \odot \rdual{M}\\
  \ep &\colon \rdual{M} \odot M \too S^n \sm U_X
\end{align*}
such that the composites
\[\xymatrix@R=14pt{
  S^n \sm M\ar[d]^-\wr&& S^n\sm M \\
  S^n\wedge U_Y\odot M\ar[r]^-{\eta\odot \id} &M\odot \rdual{M}\odot M \ar[r]^-{\id \odot \ep} & M\odot S^n\wedge U_X \ar[u]^-\wr\\
  \rdual{M}\odot S^n\wedge U_Y \ar[r]^-{\id\odot \eta}& \rdual{M}\odot M\odot \rdual{M} \ar[r]^-{ \ep\odot \id} &S^n\wedge U_X\odot \rdual{M} \ar[d]_-\wr\\
  S^n\sm \rdual{M}\ar[u]_-\wr&&S^n \sm \rdual{M}
}\]
become fiberwise homotopy equivalent to identity maps after smashing with some $S^m$.
We say that an unsectioned fibration $M\to Y\times X$ is $n$-dualizable if $M_{+(Y\times X)}$ is so.
If $Y=X=\pt$, this reduces to the original notion of $n$-duality from \S\ref{sec:lefschetz-number}.

The following theorem is proven exactly like \autoref{thm:smcompose-dual} (in fact, they are both special cases of a single theorem about bicategorical duality).

\begin{thm}[{\cite[Theorem~16.5.1]{maysig:pht}}]\label{thm:compose-duality}
  If $M\to Y\times X$ and $N\to X\times Z$ are right $n$-dualizable ex-fibrations, then so is $M\odot N \to Y\times Z$.
\end{thm}

The proof of the following theorem is also formal, when put in the right context.

\begin{thm}[{\cite[Theorem~17.3.1]{maysig:pht},~\cite[Prop.~5.3]{shulman:frbi}}]\label{thm:bco-dual}
  For any $\fn\colon Y\to X$, the base change object ${}_{\fn} X$ is right $n$-dualizable with dual $X_{\fn}$.
\end{thm}

The next theorem does require a significant amount of work, but fortunately it has already been done for us by~\cite{maysig:pht}.

\begin{thm}[{\cite[Theorem~18.5.2]{maysig:pht}}]\label{thm:mfd-cwdual}
  Suppose $M$ is a closed smooth manifold and there is an embedding of $M$ in $\mathbb{R}^n$. For any fibration $M\to X$ 
  the ex-fibration $\widecheck{M_{+X}}$ is right $n$-dualizable.
  In particular, this applies to the identity fibration $M\to M$.
\end{thm}

We write $S_M \to M$ for the ex-fibration $M_{+M} \simeq M\sqcup M \simeq S^0 \times M$ over $M$.
According to Remark \ref{rmk:ordering-convention}, we write $\widecheck{S_M}$ for $S_M$ regarded as a fibration over $\pt\times M$.
This theorem implies $\widecheck{S_M}$ is dualizable if $M$ is a closed smooth manifold.

If $M\to Y\times X$ is $n$-dualizable and we have ex-fibrations $Q\to Y\times Y$ and $P\to X\times X$ and a fiberwise ex-map $f\colon Q\odot M \to M\odot P$ (over $Y\times X$)
we define the \textbf{trace} of $f$ to be the composite
\[\xymatrix{
  \sh{S^n\sm Q } \ar[d]^-{\eta} & & \sh{S^n \sm P}
  \\
  \sh{Q \odot M \odot \rdual{M}} \ar[r]^-{f \odot \id} &
  \sh{M \odot P \odot \rdual{M}} \ar[r]^{\simeq} &
  \sh{\rdual{M}\odot M \odot P} \ar[u]_-{\ep}
}\]

\begin{rmk}
This generalizes the trace in \S\ref{sec:lefschetz-number} by replacing spaces by ex-fibrations, and also by allowing for  \emph{twisting} of both the domain 
and codomain.  The codomain twisting is used in the description of the Reidemeister trace, while the domain twisting will 
play a role in the proof of its additivity in \autoref{thm:addreidemeister}.
\end{rmk}

We will use several instances of this trace in the next sections.
The first example we will need is the Reidemeister trace.

\begin{thm}[\cite{kate:higher}]\label{thm:reidemeister_id}
  Suppose that $\widecheck{S_X}$ is right $n$-dualizable, and let $\fn\colon X\to X$ be any continuous map.
  Then the trace of the induced map $\widecheck{S_X} \to \widecheck{S_X}\odot_X X_{\fn}$ induces a map on homology:
  \[ \bbZ \cong H_0(\Lambda \star) \to H_0(\Lambda^{\fn}X) \]
  that picks out the Reidemeister trace $R(\fn)$.
\end{thm}

For purposes of this paper, the reader is free to consider this theorem a definition of the Reidemeister trace.

We also have a version of \autoref{thm:dercomposite} in this context.

\begin{thm}[{\cite[Prop.~7.5]{PS2}}]\label{thm:compose-traces2}
  Let $M\to Y\times X$ and $N\to X\times Z$ be $n$-dualizable, let $Q\to Y\times Y$, $P\to X\times X$, and $R\to Z\times Z$ be ex-fibrations, and let $f\colon Q\odot M\to M\odot P$ and $g\colon P\odot N\to N\odot R$ be fiberwise ex-maps.
  Then the following triangle commutes up to stable homotopy.
  \[ \xymatrix{ S^n \sm \sh{Q} \ar[rr]^{\tr((\id_M \odot g)\circ (f\odot\id_N))} \ar[dr]_{\tr(f)} & &
    S^n \sm \sh{R} \\
    & S^n \sm \sh{P} \ar[ur]_{\tr(g)}
    }\]
\end{thm}

There is one other general fact about traces that we will need in the following sections.
Suppose that
\[\vcenter{\xymatrix{
    Y\ar[r]^{\fn}\ar[d]_{f_1} \ar@{}[dr]|{\alpha} &
    X\ar[d]^{f_2}\\
    Y\ar[r]_{\fn} &
    X
  }}\]
is a homotopy commutative square, with specified homotopy $\alpha\colon \fn f_1\sim f_2 \fn$.
Then ${}_{\fn} X$ is right $n$-dualizable by Theorem \ref{thm:bco-dual}.  

\begin{thm}\cite[Prop.~3.3, Theorem~5.7]{PS4}\label{thm:square-trace}
  There is an induced map
  \[ h\colon Y_{f_1} \odot {}_{\fn} X \too {}_{\fn}X \odot X_{f_2}. \]
  whose trace is homotopic to the map
  \[ \Lambda^{f_1} Y_+ \to \Lambda^{f_2} X_+ \]
  which sends a path $\gamma\colon a \leadsto f_1(a)$ in $Y$ to the path $\fn(\gamma) \cdot \alpha(a)$.
\end{thm}

\section{Parametrized profunctors}  \label{sec:parametrize}

In order to generalize the proof of the additivity of the Lefschetz number to the Reidemeister trace, we will consider 
diagrams of ex-fibrations.  If $A$ and $B$ are small categories and $Y$ and $X$ are topological spaces, 
a \textbf{parametrized profunctor} from $(A,Y)$ to $(B,X)$ is an $(A \times B\op)$-indexed diagram of ex-fibrations over $Y \times  X$.
Given parametrized profunctors $M\colon (A,Y)\to (B,X)$ and $N\colon (B,X)\to (C,Z)$ we define their \textbf{composite}
$M\odot_B N$ to be a parametrized profunctor from $(A,Y)$ to $(C,Z)$ where $(M\odot_B N)(a,c)$ is the homotopy coend of 
the $(B\op\times B)$-indexed diagram
\[(b,b')\mapsto M(a,b)\odot N(b',c).\]

For a (nondegenerately) based topological space $W$ and a parametrized profunctor $M$ from $(A,Y)$ to $(B,X)$ we can define a  parametrized profunctor  $W\wedge M$ 
from $(A,Y)$ to $(B,X)$ by taking $(W\wedge M)(a,b)$ to be the ex-fibration $W\wedge (M(a,b))$ over $Y\times X$ as in the previous section.
In particular, for a small category $A$ and a space $Y$ we can define a parametrized profunctor $U_{(A,Y)}$ from $(A,Y)$ to $(A,Y)$ by $U_{(A,Y)}(a,a')=\hom_A(a,a')_+\sm U_Y$.
We also have natural equivalences 
\[(W\wedge U_{(A,Y)})\odot_A M\simeq W\wedge M \qquad M\odot_B (W\wedge U_{(B,X)})\simeq W\wedge M\] 
for any (nondegenerately) based topological space $W$. 

\begin{rmk}
As in \S\ref{sec:lefschetz}, we use a subscript on $\odot$ to indicate that we are taking a composite of profunctors.
\end{rmk}

We say that a parametrized profunctor $M$ from $(A,Y)$ to $(B,X)$ is \textbf{right $n$-dualizable} if there is a parametrized 
profunctor $\rdual M$ from $(B,X)$ to $(A,Y)$ and natural transformations $\eta\colon S^n\wedge (A,Y) \to M\odot_B\rdual{M}$ 
and $\epsilon\colon\rdual{M}\odot_A M\to S^n\wedge (B,X) $ (possibly with cofibrantly replaced domains) such that the composites
\[\xymatrix@R=14pt{
  S^n \sm M\ar[d]^-\wr&&S^n \sm  M\\
  S^n\wedge U_{(A,Y)}\odot_A M\ar[r]^-{\eta\odot \id} &M\odot_B \rdual{M}\odot_A M \ar[r]^-{\id \odot \ep} & M\odot_B S^n\wedge U_{(B,X)} \ar[u]^-\wr\\
  \rdual{M}\odot_A S^n\wedge U_{(A,Y)} \ar[r]^-{\id\odot \eta}& \rdual{M}\odot_A M\odot_B \rdual{M} \ar[r]^-{ \ep\odot \id} &S^n\wedge U_{(B,X)}\odot_B \rdual{M} \ar[d]_-\wr\\
  S^n\sm \rdual{M}\ar[u]_-\wr&&S^n \sm \rdual{M}
}\]
are naturally fiberwise homotopic to identity maps after smashing with $S^m$ for some $m$.
If we take $X$ and $Y$ to be a point $\pt$ this reduces to the duality considered in \S\ref{sec:lefschetz}.
If we take $A$ and $B$ to be the category $\tc$ with one object and no nonidentity morphisms, this reduces to the duality in \S\ref{sec:fibration}.

As in \S\ref{sec:lefschetz}, we say that a right dualizable parametrized profunctor $M$ from $(A,Y)$ to $(\tc,X)$ is \textbf{pointwise dualizable}, and a right dualizable parametrized profunctor $\Phi$ from $(\tc, Y)$ to $(B, X)$ is \textbf{absolute}.
If $\Phi$ and $M$ are parametrized profunctors from $(\tc,Y)$ to $(A,X)$ and from $(A,X)$ to $(\tc,Z)$ respectively, then their composite $\Phi\odot_A M$ is the \textbf{$\Phi$-weighted homotopy colimit} of $M$, denoted $\colim^{\Phi}(M)$.
Analogously to \autoref{thm:smpointwise-dual}, we have:

\begin{lem}[{\cite[Lemma 3.5]{PS5}}]
	%[{\cite[\autoref{PS5:thm:smcpwdual}]{PS5}}]
  A parametrized profunctor $M$ from $(A,Y)$ to $(\tc,X)$ is right dualizable if and only if the ex-fibration $M(a) \to Y\times X$ is right dualizable as an ex-fibration for each $a\in A$.
\end{lem}

For a parametrized profunctor $M$ from $(A,Y)$ to $(A,Y)$, its \textbf{shadow}, $\sh{M}_A$, is the homotopy coend of the 
$(A\op\times A)$-indexed diagram $(a,a')\mapsto \sh{M(a,a') }$.
Using the compatibility between homotopy coends and the shadow for ex-fibrations we have an isomorphism 
\[\sh{M\odot_B N}_A\cong \sh{N\odot_A M}_B\] 
analogous to those in the previous sections,
whenever both composites make sense.  As before, we write $\sh{(A,Y)}_A$ for $\sh{U_{(A,Y)}}_A$.

If $M$ is a right $n$-dualizable parametrized profunctor from $(A, Y)$ to $(B,X)$, $Q$ is a parametrized profunctor from 
$(A,Y)$ to $(A,Y)$, and $P$ is a parametrized profunctor from 
$(B,X)$ to $(B,X)$, we define the \textbf{trace} of a natural transformation $f\colon Q\odot_A M\to M\odot_B P$ to be the composite 
\[\xymatrix{
  \sh{Q\wedge S^n}_A \ar[d]^-{\eta} & & \sh{S^n \sm P}_B
  \\
  \sh{Q\odot_A M \odot_B \rdual{M}}_A \ar[r]^-{f \odot \id} &
  \sh{M \odot_B P \odot_B \rdual{M}}_A \ar[r]^{\simeq} &
  \sh{\rdual{M}\odot_A M \odot_B P}_B \ar[u]_-{\ep}
}\]
We now have the following  theorem generalizing \autoref{thm:smcompose-dual} and \autoref{thm:dercomposite}; we state it for morphisms that are additionally \emph{twisted} in the codomain, since this is what we need for the Reidemeister trace.
It is again an instance of~\cite[Prop.~7.5]{PS2}, but now occurring in the ``bicategory of diagrams in the derivator bicategory of parametrized spectra''.
Since the latter derivator bicategory is constructed in \autoref{thm:exbicat} below, here is where we really begin relying on the work of \autoref{part:derivators}.

\begin{thm}%[{\cite[Theorems 2.5, 2.6 and 11.13]{PS5}}]  
	%[{\cite[\autoref{PS5:thm:compose-duals}, \autoref{PS5:thm:compose-traces}, \autoref{PS5:thm:bicatderivlinearity}]{PS5}}]  
\label{thm:compose-traces3} 
  Let $N$ be a parametrized profunctor from $(A,Y)$ to $(B,X)$ and $M$ be a parametrized profunctor from $(B,X)$ to $(C,Z)$.
  If $M$ and $N$ are right dualizable then so is $M\odot N$.

  If $P$ is a parametrized profunctor from $(C,Z)$ to $(C,Z)$ and $g\colon N\to N$ and $f\colon M\to M\odot_B P$ are natural transformations, then 
  \[\tr(g\odot_B f)=\tr(f)\circ \tr(g).\]
  In particular, if we take $N$ to be an absolute parametrized profunctor $\Phi$ and $g=\id_\Phi$ we have 
  \[\tr(\colim^{\Phi}(f))=\tr(f)\circ \tr(\id_\Phi).\]
\end{thm}

As before, in order to be able to use this theorem we need to be able to describe $\tr(f)$ and $\tr(\id_\Phi)$.  Our approach 
will be very similar to the method used in \S\ref{sec:lefschetz}, so we start with a generalization of \autoref{thm:smpointwise-trace}.

\begin{lem} \label{thm:bipointwise-trace} %[{\cite[Lemma 11.15]{PS5}}]
	%[{\cite[\autoref{PS5:thm:derbicatomega}]{PS5}}] \label{thm:bipointwise-trace}	
  Let $M$ be a right dualizable parametrized profunctor 
  from $(A, Y)$ to $(B,X)$, $P$ be a parametrized profunctor 
  from $(B,X)$ to $(B,X)$, and  $f\colon M\to M\odot_BP$ be a natural transformation.  For any conjugacy class $[a\xto{\alpha} a]$ in $A$, the stable 
  homotopy class of the composite
  \begin{equation}
    \xymatrix{ S^n \ar[r]^-{[\alpha]} & S^n \sm \sh{(A,Y)}_A \ar[r]^-{\tr(f)} & \sh{P}_B }
  \end{equation}
  is equal to the trace of the composite
  \begin{equation}
    \xymatrix{ {}_a M \ar[r]^-{{}_\alpha M} & {}_a M \ar[r]^-{{}_a f} & {}_a(M\odot_B P) .}
  \end{equation}
  Here ${}_a M$ denotes the functor $B\op\to \ptop$ defined by ${}_a M(b) = M(a,b)$.
\end{lem}

Formally, this result follows by applying \cite[Lemma 11.15]{PS5} % \cite[\autoref{PS5:thm:derbicatomega}]{PS5}
to the derivator bicategory of parametrized spectra constructed in \autoref{thm:exbicat}.
However, for topological intuition we sketch what the proof looks like in this case.

\begin{proof}[Sketch of proof]
For each $a\in A$ we can define a parametrized profunctor $\Phi_a$ from $(\tc,Y)$ to $(A,Y)$ by 
\[\Phi_a(a')=\hom_A(a',a)_+\wedge U_Y.\]  As in \autoref{thm:smpointwise-trace}, this is absolute;
its right dual is defined by 
\[(\rdual{\Phi_a})(a')=\hom_A(a,a')_+\wedge U_Y.\]
The coevaluation and evaluation extend the maps defined in \autoref{thm:smpointwise-trace}.

The value of $\Phi_a\odot_A M$ on an element $b\in B$ is the homotopy coend of the $A\times A\op$ diagram 
\[(x,x')\mapsto \Phi_a(x)\odot M(x',b)\coloneqq A(a,x)_+\wedge U_Y\odot M(x',b) \cong A(a,x)_+\wedge M(x',b)\]
This is equivalent to the functor ${}_aM$.
In the same way we can identify
\[\alpha \odot_A \id_M \colon \Phi_a\odot_A M\to \Phi_a\odot_A M\]
with the natural transformation ${}_\alpha M \colon \,_aM\to \,_aM$ that is induced by the action of $\alpha$, and
\[{\id_{\Phi_a}\odot_A f} \colon \Phi_a\odot_A M\to  \Phi_a\odot_A M \odot_B P \]
with ${}_a f\colon \,_aM\to \,_a(M\odot_B P)$.

Now a representative $\alpha$ of a conjugacy class in $A(a,a)$ defines a map $\omega_\alpha\colon \Phi_a\to \Phi_a$
by composition.  Explicit descriptions of the coevaluation and evaluation above show that
the trace of $\omega_\alpha$ is the map $[\alpha]\colon S^0\to S^0\wedge \sh{A,Y}_A$ that takes the non-basepoint of $S^0$ to $\alpha$.
Thus, \autoref{thm:compose-traces3} implies the desired equality.
\end{proof}

\section{Linearity of the Reidemeister traces}
\label{sec:reidemeister}

We can now use  \autoref{thm:compose-traces3} and \autoref{thm:bipointwise-trace} to generalize \autoref{Lefschetz} to the 
Reidemeister trace.  Recall that in \autoref{eg:smcofibers} we showed that if $B$ is the category with two objects and one nonidentity morphism  $\alpha\colon a\to b$
and $\Phi\colon B\op\to \ptop$ is defined by $\Phi(b)=S^0$ and $\Phi(a)=\pt$ then $\Phi$ is absolute.  
Since a based space is an ex-fibration over $\pt$,  $\Phi$ is also a parametrized profunctor $\Phi$ from $(\tc, \pt)$ to $(B,\pt)$.  Further, the coevaluation and evaluation for the 
profunctor $\Phi$ are also a coevaluation and evaluation for the parametrized profunctor $\Phi$.

\begin{thm}\label{thm:addbicat}  If $M_a$ and $M_b$ are right dualizable ex-fibrations over $\pt\times X$, $P$ is an 
ex-fibration over $X\times X$ and $f_a\colon M_a\to M_a\odot P$,  $f_b\colon M_b\to M_b\odot P$, and $i\colon M_a\to M_b$ are fiberwise maps 
so that the diagram 
\[\xymatrix{M_a\ar[r]^i\ar[d]^{f_a}&M_b\ar[d]^{f_b}\\
M_a\odot P\ar[r]^{i}&M_b\odot P}\]  commutes then 
\[\tr(\colim^\Phi (f))=\tr(f_b)-\tr(f_a).\]
\end{thm}

\begin{proof} As shown in \autoref{eg:smcofibers}, $\Phi$ is absolute and $\pi_0^s(\sh{B})\cong \bbZ\oplus \bbZ$.  We saw  in the proof of \autoref{Lefschetz} 
that
$\tr(\id_\Phi)$ is determined by morphisms
$\phi_a, \phi_b\colon S^n\to S^n$, which we calculated to be $\phi_a=-1$ and $\phi_b=1$.
Thus, for any pointwise dualizable parametrized profunctor $M$ from $(B,\pt)$ to $(\tc,X )$ and endomorphism
$f\colon M\to M\odot_\tc P$, the desired statement follows from the linearity formula 
\[\tr(\colim ^\Phi(f))=\phi_a \cdot \tr(f_a)+\phi_b\cdot \tr(f_b).\qedhere\]
\end{proof}

\begin{rmk}
This theorem also holds if $M_a$ and $M_b$ are right dualizable ex-fibrations over $Y\times X$.  In the proof we replace
$\Phi$ 
with the parametrized profunctor from $(\tc,\pt)$ to $(B, Y)$ whose value on $b$ is $S_Y$ 
and on $a$ is the identity map of $Y$ regarded as an ex-fibration (no disjoint section is added).
In this case, the $\Phi$-weighted colimit of $M$ is the fiberwise  homotopy pushout of the diagram on the left below.
\begin{equation}
\vcenter{\xymatrix{S^0\wedge M_a\ar[r]\ar[d]&S^0\wedge M_b\\
  0\wedge M_a}}\hspace{2cm}
\vcenter{\xymatrix{M_a\ar[r]\ar[d]&M_b\\
Y\times X
  }}
\end{equation}
More explicitly, $\colim^\Phi M$ is the homotopy pushout of the diagram on the right above; we call this the \textbf{fiberwise mapping cone} on $i$ and denote it by $M(i)$.
\end{rmk}

Now we can state our additivity theorem for the Reidemeister trace.

\begin{thm}\label{thm:addreidemeister}  Suppose $Y$ and $X$ are closed smooth manifolds or compact ENRs and $i\colon Y\to X$, 
$f\colon X\to X$ and $g\colon Y\to Y$ are continuous maps so that $i\circ f=g\circ i$.  Then 
\[R(f)-i(R(g))=R_{X|Y}(f),\] where $R_{X|Y}(f)$ is the trace of the induced map
$\widecheck{M(i)}\to \widecheck{M(i)}\odot X_f$
 and   $i\colon \Lambda^{g} Y\xto{}\Lambda^f X$  is induced by $i$.
\end{thm}

We claim that this is a refinement of \cite[Theorem~3.2.1]{ferrario}.
To see this, let $j\colon X\to X/Y$ be the quotient map.
If $i\colon Y\xto{} X$ is a cofibration
of compact ENR's or closed smooth manifolds, then $j(R_{X|Y}(f))=R(f/Y)$ \cite{kate:relative} and $ji(R(f|_Y))$ is equal to the Lefschetz number $L(f|_Y)$.
Thus, we recover Ferrario's result
\[  j(R(f))- L(f|_Y)= R(f/Y).\]

\begin{proof} %[Proof of \autoref{thm:addreidemeister}]  
Regarding $f\colon S_X\to S_X$ 
as a fiberwise map over $f\colon X\to X$, we have a map $f\colon \widecheck{S_X}\to \widecheck{S_X}\odot X_f$
over $X$  and a similar map 
$g\colon \widecheck{S_Y}\to \widecheck{S_Y}\odot Y_{g}$ over $Y$.  The map $i\colon Y\to X$ induces a fiberwise map $i_!(S_Y)\to S_X$,
and by \autoref{thm:square-trace}   the equality $i\circ g=f\circ i$
defines a fiberwise map $k\colon Y_{g}\odot \,_iX\to \,_iX\odot X_f$ over $Y\times X$.  We compose these last two maps to define  
a twisted endomorphism of $i_!(\widecheck{S_Y})$:
 \[\tilde{g}\colon \widecheck{S_Y}\odot \,_iX\xto{g\odot \id}   \widecheck{S_Y}\odot Y_{g}\odot \,_iX\xto{\id \odot k} \widecheck{S_Y}\odot \,_iX\odot X_f.\]
This gives a commutative diagram 
\[\xymatrix{
i_!(\widecheck{S_Y})\ar[r]\ar[d]^{\tilde{g}}&\widecheck{S_X}\ar[d]^f\\
i_!(\widecheck{S_Y})\odot X_f\ar[r]&\widecheck{S_X}\odot X_f
}\]
of ex-fibrations over $X$.  Let $M$ from $(B,\star)$ to $(\bbone, X)$ be the parametrized profunctor 
where $M_a=i_!(\widecheck{S_Y})$ and $M_b=\widecheck{S_X}$.  The diagram above gives us a natural transformation
$f\colon M\to M\odot X_f$.  
For $\Phi$ as defined before \autoref{thm:addbicat}, the $\Phi$-weighted colimit of $M$ is  $\widecheck{M(i)}$, the 
fiberwise mapping cone on $i$.  

By assumption $\widecheck{S_X}$ and $\widecheck{S_Y}$ are right $n$-dualizable for some integer $n$.   By \autoref{thm:bco-dual} the object $_iS$ 
 is right dualizable.  Then \autoref{thm:compose-traces2} implies $i_!\widecheck{S_Y}\cong \widecheck{S_Y}\odot \,_iX$ is right dualizable. 
 Using \autoref{thm:addbicat} we see 
  \[ \tr(h\colon \widecheck{Mi}\to \widecheck{Mi}\odot X_f) = \tr(f\colon \widecheck{S_X}\xto{}\widecheck{S_X}\odot X_f) - \tr(\tilde{g}\colon i_!(\widecheck{S_Y})\xto{}i_!(\widecheck{S_Y})\odot X_f). \]
 In \autoref{thm:reidemeister_id} we saw that  the trace of $f\colon \widecheck{S_X}\to 
  \widecheck{S_X}\odot X_f$ is the Reidemeister trace of $f$.
Theorems \ref{thm:compose-traces2} and \ref{thm:square-trace} 
imply
\[\tr(\tilde{g}\colon i_!(\widecheck{S_Y})\xto{}i_!(\widecheck{S_Y})\odot X_f)=i(\tr(g\colon \widecheck{S_Y}\to \widecheck{S_Y}\odot Y_{g} ))=i(R(g)).\]
Finally, $R_{X|Y}(f)$ is defined to be the 
trace of $\widecheck{Mi}\xto{M(f)} \widecheck{Mi}\odot X_f$.
\end{proof}

\part{Indexed monoidal derivators}
\label{part:derivators}

It remains to set up the abstract context that encapsulates all the fibrant and cofibrant replacements, so that \autoref{thm:dercomposite}, \autoref{thm:compose-traces2}, and \autoref{thm:compose-traces3} are, like~\cite[Theorem~5.2]{PS4}, all instances of~\cite[Prop.~7.5]{PS2}.
% In~\cite{PS4}, the necessary work had already been done for us by~\cite{maysig:pht,PS3}.
% In this paper, the necessary work
This is \emph{almost} completed by the companion paper~\cite{PS5}, which studies linearity formulas for trace-like invariants in categorical generality.
The examples given therein suffice for the additivity of the Lefschetz number (\S\ref{sec:lefschetz}), but for the additivity of Reidemeister trace (\S\S\ref{sec:parametrize}--\ref{sec:reidemeister}) we need to combine this theory with that of~\cite{maysig:pht,PS3}.
Therefore, in this part of the paper, we will show that the parametrized profunctors used in \S\S\ref{sec:fibration}--\ref{sec:reidemeister}
form a \emph{derivator bicategory}, so that the results of \cite{PS5} can then be applied in this context.
This provides the foundation for the results in \S\ref{sec:parametrize}.

Relative to \autoref{part:topological}, we now make one change of context: instead of working with ex-fibrations we will consider parametrized spaces.
A \textbf{parametrized space} over a base space $X$ is an arbitrary space $E$ with 
maps $s\colon X\to E$ and $p\colon E\to X$ so that $p\circ s=\id$. This is less restrictive than an ex-fibration; we think of an ex-fibration as a very nice parametrized 
space that is especially well behaved homotopically.  This change reflects the shift from choosing specific well behaved examples in the previous sections to the more systematic 
approach using model categories in the remaining sections.

\section{Background}
\label{sec:deriv-bicat}

We start by recalling some relevant definitions from \cite{PS5}.
We will not recall the definition of a \emph{derivator}; all we need to know is that it is a refinement of a homotopy category that also includes information about diagram categories and homotopy limits and colimits.
If $\D$ is a derivator and $A$ is a small category, we write $\D(A)$ for the corresponding homotopy category of $A$-shaped diagrams.
We refer to object of $\D(A)$ as ($A$-shaped) \emph{coherent diagrams} in $\D$.
See~\cite{PS5} or~\cite{groth:ptstab,gps:additivity} for more details.

The basic context used in~\cite{PS5}  is a derivator bicategory.
By definition, a \textbf{derivator bicategory} $\dW$ consists of the following data.
  \begin{itemize}
  \item A collection of objects $R$, $S$, $T$, $\ldots$.
  \item For each pair of objects $R$ and $S$ a derivator $\dW({R,S})$.
    We think of the category $\dW(R,S)(A)$ as the homotopy category of $A$-shaped diagrams in the ``hom-category'' $\dW(R,S)$.
  \item For each triple of objects $R$, $S$, and $T$, a morphism of derivators
    \[\odot \colon \dW({R,S})\times \dW({S,T})\to \dW({R,T}).\]
    That is, we have a pseudonatural transformation between 2-functors $\cCat\op\to\cCAT$, which has components
    \[ \dW({R,S})(A)\times \dW({S,T})(A)\to \dW({R,T})(A).\]
  \item We require these morphisms $\odot$ to be cocontinuous in each variable separately \cite[Definition~3.19]{gps:additivity}.
  \item For each object $R$, a morphism of derivators $\tc\to \dW(R,R)$ (hence an object $\lI_{R,A}\in \dW(R,R)(A)$, varying pseudonaturally in $A\in\cCat$).
  \item Natural unit and associativity isomorphisms, i.e.\ invertible modifications
    \begin{equation}
      \vcenter{\xymatrix{
          \dW(R,S) \times \dW(S,T) \times \dW(T,U)\ar[r]^-{\id\times\odot}\ar[d]_{\odot\times\id} \drtwocell\omit{\cong} &
          \dW(R,S) \times \dW(S,U)\ar[d]^\odot\\
          \dW(R,T) \times \dW(T,U)\ar[r]_-{\odot} &
          \dW(R,U)
        }}
    \end{equation}
    \begin{equation}
      \xymatrix{\dW(R,S) \ar[r]^-{(\id,\lI)} \drlowertwocell{\cong} &
        \dW(R,S) \times \dW(S,S)\ar[d]^\odot \\
        & \dW(R,S)
      }
      \qquad
      \xymatrix{\dW(R,S) \ar[r]^-{(\lI,\id)} \drlowertwocell{\cong} &
        \dW(R,R) \times \dW(R,S)\ar[d]^\odot \\
        & \dW(R,S).
      }
    \end{equation}
  \item The usual pentagon and unit axioms for a bicategory hold.
  \end{itemize}
  A derivator bicategory is \textbf{closed} if the morphisms $\odot$ participate in a two-variable adjunction of derivators~\cite[Definition~8.1]{gps:additivity}.
  A \textbf{shadow} on a derivator bicategory \dW consists of a derivator \dT and cocontinuous morphisms of derivators
  \[\sh{-} \colon \dW(R,R) \xto{} \dT\]
  for each object $R$, together with invertible modifications
  \begin{equation}
    \vcenter{\xymatrix{
        \dW(R,S)\times \dW(S,R)\ar[rr]^{\cong} \ar[d]_{\odot} \drrtwocell\omit{\cong} &&
        \dW(S,R)\times \dW(R,S)\ar[d]^{\odot}\\
        \dW(R,R)\ar[r]_-{\sh{-}} &
        \dT\ar@{<-}[r]_-{\sh{-}} &
        \dW(S,S)
      }}
  \end{equation}
  satisfying the usual compatibility axioms for a shadow (\cite[Defn.~4.1]{PS2}).
The \emph{underlying bicategory} of a derivator bicategory $\dW$ is the bicategory $\bW$ with the same objects as $\dW$ and with morphisms
given by $\bW(R,S) = \dW(R,S)(\tc)$.

The remaining sections of this paper are a proof of the following theorem.
\begin{repthm}{thm:exbicat}
  There is a closed derivator bicategory \Ex, with a shadow valued in the homotopy derivator of spectra.
  Its objects are compactly generated spaces, its hom $\Ex(R,S)$ is the homotopy derivator of the model category of spectra parametrized over $R\times S$, and its underlying ordinary bicategory is the one constructed in~\cite[Chapter~17]{maysig:pht}.
\end{repthm}
Before we continue to the proof, we will first recall the generalizations of several of the essential results in the previous 
sections to derivator bicategories.

\begin{thm}[{\cite[Theorem 11.7]{PS5}}]\label{thm:derivbicat-prof}
	%[{\cite[\autoref{PS5:thm:derivbicat-prof}]{PS5}}]\label{thm:derivbicat-prof}
  Given a derivator bicategory \dW, we can construct a derivator bicategory $\dprof(\dW)$, with underlying ordinary bicategory denoted $\prof(\dW)$.  The latter is described as follows:
  \begin{itemize} 
  \item An object is a pair $(A,R)$ where $A\in\cCat$ and $R$ is an object of $\dW$.
  \item The hom category from $(A,R)$ to $(B,S)$ is $\dW({R,S})(A\times B\op)$.
  \item The composition functors are constructed by taking homotopy coends.
  \item The unit object of $(A,R)$, denoted $\lI_{(A,R)}$, is a diagram whose value at $(a,a')$ is the coproduct of $\hom_A(a',a)$ copies of the unit $\lI_R$.
  \end{itemize}
  For the derivator bicategory $\dprof(\dW)$, the hom-derivators are defined by
  \[\dprof(\dW)((A,R),(B,S))(C) = \dW(R,S)(A\times B\op\times C)\]
  and the composition and units are defined analogously.
  If \dW is closed, then so is $\dprof(\dW)$ (and hence also $\prof(\dW)$).

  Finally, if \dW has a shadow valued in a derivator \dT, then so does $\dprof(\dW)$, defined by taking homotopy coends.
  It follows that $\prof(\dW)$ also has a shadow valued in $\dT(\bbone)$.
\end{thm}

We follow the definitions in the previous section and say 
a coherent diagram $X\in\dW(R,S)(A)$ is \textbf{pointwise dualizable} if it is right dualizable when regarded as a 1-cell from $(A,R)$ to $(\bbone,S)$ in $\prof(\dW)$.
Similarly, a coherent diagram $\Phi\in\dW(R,R)(A\op)$ is \textbf{absolute} if it is right dualizable when regarded as a 1-cell from $(\bbone,R)$ to $(A,R)$ in $\prof(\dW)$.

The following statements are the generalizations of \autoref{thm:compose-traces3} and \autoref{thm:bipointwise-trace} to derivator bicategories; once we have \autoref{thm:exbicat} then they will imply \autoref{thm:compose-traces3} and \autoref{thm:bipointwise-trace}.

\begin{thm}[{\cite[Theorem 11.13]{PS5}}]  
	%[{\cite[\autoref{PS5:thm:bicatderivlinearity}]{PS5}}]  
  \label{thm:bicatderivlinearity}
  If $X\in\dW(R,S)(A)$ is pointwise dualizable and $\Phi\in\dW(R,R)(A\op)$ is absolute, then for any $f\colon X\to X$ we have
  \[ \tr(\colim^\Phi (f)) = \tr(f) \circ \tr(\id_\Phi). \]
\end{thm}

\begin{lem}[{\cite[Lemma 11.15]{PS5}}]  \label{thm:derbicatomega}
	%[{\cite[\autoref{PS5:thm:derbicatomega}]{PS5}}]  \label{thm:derbicatomega}
  If $X\in\dW(R,S)(A)$ is pointwise dualizable and $f\colon X\to X$, then for any conjugacy class $[a\xto{\alpha} a]$ in $A$, the composite
  \begin{equation}
    \xymatrix{ \sh{R} \ar[r]^-{[\alpha]} & \sh{(A,R)} \ar[r]^-{\tr(f)} & \sh{S} }
  \end{equation}
  is equal to the trace in $\dW(R,S)(\bbone)$ of the composite
  \begin{equation}
    \xymatrix{ X_a \ar[r]^-{X_\alpha} & X_a \ar[r]^-{f_a} & X_a .}
  \end{equation}
\end{lem}

\section{Indexed monoidal derivators}
\label{sec:indexed}

We now introduce an intermediate structure which we will use to construct the derivator bicategory of parametrized spectra.
Recall that if $\bS$ is a category, an $\bS$-\textbf{indexed category} is a pseudofunctor $\sC$ from $\bS\op$ to $\cCAT$, and an $\bS$-\textbf{indexed monoidal category} is a pseudofunctor from $\bS\op$ to $\cMONCAT$.
Thus, for each each object $R\in \bS$ we have a (monoidal) category $\sC^R$, and for each morphism $f\colon R\to S$ in $\bS$ we have a (monoidal) functor $\pb{f}\colon \sC^S \to \sC^R$.

In an indexed monoidal category, we write the tensor product in each fiber $\sC^R$ as $\otimes$, or $\otimes_R$ for emphasis if needed.
If \bS has finite products, as we will henceforth assume, then there is also an \emph{external} tensor product
$\boxtimes \colon  \sC^R \times \sC^S \to \sC^{R\times S}$
defined by
\[ X \boxtimes Y = \pi_S^* X \otimes_{R\times S} \pi_R^*Y. \]
The external product is associative and unital in a suitable sense, and we can recover the fiberwise products from it as $X\otimes_R Y = (\Delta_R)^*(X\boxtimes Y)$;
see~\cite{shulman:frbi} (where unfortunately the meanings of $\otimes$ and $\boxtimes$ are reversed).

In order to define a bicategory from an indexed monoidal category, we need a little more structure.

\begin{defn}
  We say that \sC\ has \textbf{\bS-indexed coproducts} if
  \begin{enumerate}
  \item Each reindexing functor $\pb{f}$ has a left adjoint $\pf{f}$, and 
  \item For any pullback square
    \[\xymatrix@-.5pc{A \ar[r]^f\ar[d]_h & B \ar[d]^g\\
      C\ar[r]_k & D}\]
    in \bS, the composite
    \[\pf{f} \pb{h} \too \pf{f}\pb{h}\pb{k}\pf{k} \too[\iso] \pf{f}\pb{f}\pb{g}\pf{k} \too \pb{g}\pf{k}\]
    is an isomorphism (the \emph{Beck-Chevalley condition}).
  \end{enumerate}
  If \sC\ is symmetric monoidal, we say that \textbf{$\ten$ preserves
    indexed coproducts} (in each variable separately), or that the
  \textbf{projection formula holds}, if
  \begin{enumerate}[resume]
  \item for any $f\maps A\to B$ in \bS\ and any $M\in\sC^B$,
    $N\in\sC^A$, the canonical map
    \[\pf{f}(\pb{f}M \otimes N) \to \pf{f}(\pb{f}M \otimes \pb{f} \pf{f} N)
    \cong \pf{f} \pb{f} (M \otimes \pf{f} N) \to M \otimes \pf{f}N
    \]
    is an isomorphism.
      \label{item:prescoprod}
  \end{enumerate}
\end{defn}

\begin{lem}[{\cite[Lemma~3.2]{PS3}}]\label{thm:altproj}
  The projection formula is equivalent to asking that for $f\maps A\to B$, $M\in \sC^C$, and $N\in\sC^A$, the canonical map
  \begin{align}
    \pf{(\id_C\times f)}(M\boxtimes N)
    &\to \pf{(\id_C\times f)}(M\boxtimes \pb{f}\pf{f}N)\\
    &\cong \pf{(\id_C\times f)}\pb{(\id_C\times f)}(M\boxtimes \pf{f}N)\\
    &\to M\boxtimes \pf{f}N
  \end{align}
  is an isomorphism.
\end{lem}

In \cite{shulman:frbi} it is shown that if $\sC$ is an indexed monoidal category with indexed coproducts preserved by $\ten$, then there is a bicategory $\calDi \sC\bS$, whose 0-cells are the objects of $\bS$, and whose hom-categories are
$\calDi \sC\bS(R,S)=\sC_{R\times S}$.  The compositions and units are defined for $M\in \calDi \sC\bS(R,S)$ and $N\in \calDi \sC\bS(S,T)$ by
\begin{align*}
  M\odot N &= \pf{(\id_R\times\pi_S\times\id_T)}
  \pb{(\id_R\times \Delta_S\times\id_T)}(M \boxtimes N) \quad\text{and}\\
  \lI_R &= \pf{(\Delta_R)} \pb{\pi_R} (\lS)
\end{align*}
In~\cite{PS3} it is shown that if $\sC$ is symmetric, then $\calDi \sC\bS$ has a shadow with values in $\sC^\star$, where $\star$ is the terminal object of $\bS$, defined for $M\in \calDi \sC\bS(R,R)$ by
\[ \sh{M} = \pf{(\pi_R)}\pb{(\Delta_R)} M.\]

Unfortunately, the examples arising in homotopy theory do not always satisfy the Beck-Chevalley condition for all pullback squares in \bS.
However, they do satisfy it for squares of the following form (which are all automatically pullbacks in any category with products):
\begin{equation}\label{eq:hopb}
  \vcenter{\xymatrix{
      A\times C\ar[r]^-{f\times \id}\ar[d]_-{\id\times g} &
      B\times C\ar[d]^-{\id\times g} &
      A\ar[r]^-{(\id,f)}\ar[d]_f &
      A\times B\ar[d]^-{f\times \id} &
      A\ar[r]^-{\Delta}\ar[d]_-{\Delta} &
      A\times A\ar[d]^-{\id\times \Delta}\\
      A\times D\ar[r]_-{f\times \id} &
      B\times D &
      B\ar[r]_-{\Delta_B} &
      B\times B &
      A\times A\ar[r]_-{\Delta \times \id} &
      A\times A \times A
      }}
\end{equation}
In~\cite{PS3} we defined an indexed category to have \textbf{indexed homotopy coproducts} if each $\pb{f}$ has a left adjoint $\pf{f}$ satisfying the Beck-Chevalley condition for pullback squares of these three forms, along with any square obtained from these by taking cartesian products with a fixed object, and showed that this suffices for the construction of a bicategory.
We will refer to these special pullback squares as \textbf{homotopy pullback squares}; this is not an egregious abuse of terminology since all such squares \emph{are} homotopy pullback squares in any reasonable homotopy theory (for instance, in any derivator).

To generalize this construction to derivators, we start by defining an \textbf{$\bS$-indexed monoidal derivator} to be a pseudofunctor
\[\dV \colon \bS\op \to \cMONDER_{\mathrm{cc}},\]
where $\cMONDER_{\mathrm{cc}}$ is the 2-category of monoidal derivators and cocontinuous strong monoidal morphisms~\cite[\S3]{gps:additivity}.
Thus, associated to each object $R$ of $\bS$ we have a monoidal derivator denoted by $\dV^R$ and for each morphism $f\colon R\to S$ in $\bS$ 
we have a cocontinuous strong monoidal morphism of derivators
\[\dV^S\to \dV^R\] denoted by 
$\pb{f}$. 
We say that \dV is \textbf{symmetric} if each $\dV^R$ is a symmetric monoidal derivator and each $\pb{f}$ is a symmetric monoidal morphism.

If $A$ is a small category, we denote the value of the derivator $\dV^R$ on $A$ by $\dV^R(A)$.  
For each small category $A$, if we define  $(\dV^{-}(A))^R=\dV^R(A)$, we have an ordinary $\bS$-indexed monoidal category $\dV^{-}(A)$.
Thus, we can also think of an $\bS$-indexed monoidal derivator as a derivator consisting of $\bS$-indexed monoidal categories.

\begin{defn}\label{thm:derindcoprod}
  An $\bS$-indexed derivator \dV has \textbf{$\bS$-indexed homotopy coproducts} if
  \begin{enumerate}
  \item each morphism of derivators $\pb{f} \colon \dV^S \to \dV^R$ has a left adjoint $\pf{f}$, and\label{item:dic1}
  \item these adjoints satisfy the Beck-Chevalley condition for homotopy pullback squares in $\bS$.\label{item:dic2} 
  \end{enumerate}
  If $\dV$ is additionally monoidal, we say that its indexed coproducts are \textbf{preserved by $\ten$}, or that the \textbf{projection formula} holds, if it holds levelwise for each $\dV^{-}(A)$.
\end{defn}

\autoref{thm:derindcoprod}\ref{item:dic2} is equivalent to asking that that for each $A\in\cCat$, the indexed monoidal category $\dV^{-}(A)$ has $\bS$-indexed homotopy coproducts.

\begin{thm}\label{thm:der-bicat-indexed}
From an $\bS$-indexed monoidal derivator \dV with indexed homotopy coproducts preserved by $\ten$, we can define a derivator bicategory $\calDi \dV\bS$ whose 0-cells are the objects of \bS, with
  \[\calDi \dV\bS(R,S) = \dV^{R\times S},
  \]
  and with composition and units defined by
  \begin{align*}
    M\odot N &= \pf{(\id_R\times\pi_S\times\id_T)}
    \pb{(\id_R\times \Delta_S\times\id_T)}(M \boxtimes N) \quad\text{and}\\
    \lI_R &= \pf{(\Delta_R)} \pb{\pi_R} (U)
  \end{align*}
  Moreover, if \dV is symmetric, then $\calDi \dV\bS$ has a shadow with values in $\dV^\star$, defined by
  \[ \sh{M} = \pf{(\pi_R)}\pb{(\Delta_R)} M.\]
\end{thm}
\begin{proof}
  Note that for any derivator bicategory $\dW$ and any $A\in \cCat$, we have an ordinary bicategory $\dW(A)$, with hom-categories $\dW(A)(R,S) = \dW(R,S)(A)$.
  (The case $A=\bbone$ yields the underlying ordinary bicategory of $\dW$.)
  Similarly, for any $u\colon A\to B$ in $\cCat$, we have an induced pseudofunctor $\dW(B) \to \dW(A)$ that is the identity on objects, and for any natural transformation in $\cCat$ we have an \emph{icon} in the sense of~\cite{lack:icons}.
  Together these fit together into a pseudofunctor from $\cCat\op$ to the 2-category of bicategories, identity-on-objects pseudofunctors, and icons, which encapsulates all the data of $\dW$.
  Thus, a derivator bicategory can equivalently be defined as such a pseudofunctor such that each induced pseudofunctor $\dW(R,S)$ is a derivator and the composition morphisms are cocontinuous in each variable.

  Now we have observed above that an $\bS$-indexed monoidal derivator $\dV$ can be regarded as a pseudofunctor from $\cCat\op$ to $\bS$-indexed monoidal categories.
  Thus, the functoriality of the construction in~\cite{shulman:frbi} automatically yields all the data of a derivator bicategory.
  Moreover, each $\dV^{R\times S}$ is a derivator by assumption, while cocontinuity of composition follows from the cocontinuity of the fiberwise monoidal structures (as assumed in the definition of monoidal derivator), the assumed cocontinuity of the morphisms $\pb{f}$, and the fact that $\pf{f}$ is cocontinuous as it is a left adjoint~\cite[Prop.~2.9]{groth:ptstab}.
  Thus, we have a derivator bicategory.

  The construction of shadows in~\cite{PS3} applied at each $A\in\cCat$ yields functors $\calDi\dV\bS(R,S)(A) \to \dV^\star(A)$ that satisfy the shadow axioms, so it remains to show that these assemble into cocontinuous morphisms of derivators.
  However, this is immediate since $\pb{(\Delta_R)}$ and $\pf{(\pi_R)}$ are such.
\end{proof}

\section{Closed structures}\label{sec:indclosed}
We now investigate closedness of $\calDi \dV\bS$, which requires further assumptions on $\dV$.
We begin with the corresponding structure for ordinary indexed categories.

\begin{defn}\label{def:indexed-products}
  An \bS-indexed category \sC\ has \textbf{\bS-indexed (homotopy) products} if each reindexing functor $\pb{f}$ has a \emph{right} adjoint $\copf{f}$ that satisfies the Beck-Chevalley condition for all (homotopy) pullback squares in \bS.
\end{defn}

A standard calculation with mates implies that if the functors $\pb{f}$ have both left and right adjoints, then the left adjoints satisfy the Beck-Chevalley condition if and only if the right adjoints do.

\begin{defn}
  An \bS-indexed symmetric monoidal category \sC is \textbf{closed} if
  \begin{enumerate}
  \item each monoidal category $\sC^R$ is closed, with internal-homs $\RHD$, and
  \item the reindexing functors $\pb{f}$ are closed monoidal, i.e.\ for $f\colon R\to S$ and $M,N\in\sC^S$ the canonical map
    \begin{align}
      \pb{f}(M\RHD N)
      &\to \pb{f}M \RHD (\pb{f}(M\RHD N) \otimes \pb{f} M)\\
      &\cong \pb{f}M \RHD \pb{f}((M\RHD N) \otimes M)\\
      &\to (\pb{f}M \RHD \pb{f} N)
    \end{align}
    is an isomorphism.
  \end{enumerate}
\end{defn}

Recall that we assume \bS has finite products.

\begin{lem}\label{thm:intextclosed}
  If \sC has \bS-indexed homotopy products, it is closed if and only if
  \begin{enumerate}
  \item each external product functor $\boxtimes\colon \sC^R\times \sC^S \to \sC^{R\times S}$ is part of a two-variable adjunction, with adjoints
    $\xrhd\colon (\sC^S)\op \times \sC^{R\times S} \to \sC^R$, and\label{item:closed1}
  \item for each $f\colon R\to T$, $M\in\sC^S$, and $N\in\sC^{T\times S}$, the canonical map
    \begin{align}
      \pb{f}(M\xrhd N)
      \to (M \xrhd \pb{(f\times \id_S)} N)
    \end{align}
    is an isomorphism.\label{item:closed2}
  \end{enumerate}
\end{lem}
\begin{proof}
  This is essentially~\cite[Propositions 13.9 and 13.15]{shulman:frbi}; the only difference is the Beck-Chevalley conditions assumed.
  Proposition 13.15 of \textit{ibid} assumes indexed products (``a strongly BC $*$-fibration'') for one direction of the equivalence and a weaker form of indexed homotopy coproducts (``a weakly BC $*$-fibration'') for the other, but in fact indexed homotopy coproducts suffice for both directions.
\end{proof}

\begin{lem}\label{thm:closed-proj}
  If \sC has \bS-indexed homotopy products and coproducts and satisfies \autoref{thm:intextclosed}\ref{item:closed1}, then it satisfies \autoref{thm:intextclosed}\ref{item:closed2} (hence is closed) if and only if its indexed coproducts are preserved by $\ten$.
\end{lem}
\begin{proof}
  The canonical map in \autoref{thm:altproj} is a mate of that in \autoref{thm:intextclosed}\ref{item:closed2}, in such a way that one is an isomorphism if and only if the other is.
\end{proof}

It is shown in~\cite{shulman:frbi} that in the situation of \autoref{thm:closed-proj}, the bicategory $\calDi \sC\bS$ is closed.
We now generalize this to derivators.

\begin{defn}
  An \bS-indexed derivator \dV has \textbf{\bS-indexed homotopy products} if each derivator morphism $\pb{f}$ has a right adjoint morphism $\copf{f}$ that satisfies the Beck-Chevalley condition for all homotopy pullback squares in \bS.
\end{defn}

\begin{defn}
  An \bS-indexed symmetric monoidal derivator \dV is \textbf{closed} if each monoidal derivator $\dV^R$ is closed, with internal-homs $\RHD$, and the reindexing functors $\pb{f}$ are closed monoidal.
\end{defn}

Recall that closedness of $\dV^R$ means that the derivator morphism $\otimes \colon \dV^R \times \dV^R \to \dV^R$ is part of a two-variable adjunction.
In this case, the second condition for closedness of \dV is equivalent to closedness of each indexed monoidal category $\dV^{-}(A)$.

\begin{lem}\label{thm:intext-der}
  If an indexed monoidal derivator \dV has indexed homotopy products, then it is closed if and only if
  \begin{enumerate}
  \item each derivator morphism $\boxtimes\colon \dV^R \times \dV^S \to \dV^{R\times S}$ is part of a two-variable adjunction, and\label{item:dercl1}
  \item the condition of \autoref{thm:intextclosed}\ref{item:closed2} holds levelwise for each $\dV^{-}(A)$.
  \end{enumerate}
\end{lem}
\begin{proof}
  The definitions of $\RHD$ and $\xrhd$ in terms of each other from~\cite[Proposition~13.9]{shulman:frbi} can be applied verbatim to derivator morphisms.
  (More precisely, we apply~\cite[Lemma~8.13]{gps:additivity}.)
  The rest then follows from \autoref{thm:intextclosed} applied to each $\dV^{-}(A)$.
\end{proof}

\begin{lem}\label{thm:clproj-der}
  If an indexed monoidal derivator \dV has indexed homotopy products and coproducts and satisfies \autoref{thm:intext-der}\ref{item:dercl1}, then it is closed if and only if its indexed homotopy coproducts are preserved by $\ten$.
\end{lem}
\begin{proof}
  Apply \autoref{thm:closed-proj} to each $\dV^{-}(A)$.
\end{proof}

\begin{thm}\label{thm:der-bicat-indexed-closed}
  If \dV is a closed $\bS$-indexed monoidal derivator with indexed homotopy products and indexed homotopy coproducts, then the derivator bicategory $\calDi \dV\bS$ from \autoref{thm:der-bicat-indexed} is closed.
\end{thm}
\begin{proof}
  Recall that a derivator bicategory \dW is closed if each derivator morphism $\odot\colon \dW(R,S)\times \dW(S,T) \to \dW(R,T)$ is part of a two-variable adjunction.
  However, the construction of its adjoints in~\cite[Proposition~17.5]{shulman:frbi} can be applied verbatim in to derivator morphisms (again, technically invoking~\cite[Lemma~8.13]{gps:additivity}).
\end{proof}

\section{Indexed monoidal model categories}
\label{sec:indexed-mmc}

Finally, we need to be able to construct indexed monoidal derivators, which we do by introducing a notion of \emph{indexed monoidal model category}.
However, there is a slightly subtle twist.
We could define an $\bS$-indexed monoidal model category to be a pseudofunctor from $\bS\op$ to the 2-category of monoidal model categories, but this would not include the desired example of parametrized spectra.
Instead we ask that the \emph{external} form of the monoidal product is a two-variable Quillen left adjoint.

There is an additional problem in that an isomorphism relating composites of left and right adjoints does not automatically descend to homotopy categories, even if all the functors involved have derived versions.
For instance, a Beck-Chevalley isomorphism $\pf{f} \pb{h} \cong \pb{g}\pf{k}$ does not necessarily induce an isomorphism $\bL\pf{f} \circ \bR\pb{h} \cong \bR\pb{g} \circ \bL\pf{k}$ of derived functors.
However, a transformation of this sort does induce a \emph{not-necessarily-invertible} transformation between composites of derived functors called its \emph{derived transformation}.
The main result of~\cite{shulman:dblderived} is that the derived transformation of a mate (such as a Beck-Chevalley transformation) is the corresponding mate at the level of derived functors.
Thus, to show that a derived Beck-Chevalley isomorphism (for instance) holds, it suffices to analyze the formula for the derived transformation given in~\cite{shulman:dblderived} and show that it is an isomorphism.
This analysis depends on the particular example in question, so in the general definition to follow we simply assume that certain derived transformations are isomorphisms.
Fortunately, in our desired example of parametrized spectra, the necessary analysis has already been done by~\cite{maysig:pht}.

\begin{defn}\label{def:immc}
  An $\bS$-\textbf{indexed (symmetric) monoidal model category} consists of:
  \begin{enumerate}
  \item A pseudofunctor $\sC\colon \bS\op\to \cRMODEL_{\mathrm{cc}}$, where $\cRMODEL_{\mathrm{cc}}$ is the 2-category of model categories and right Quillen functors that preserve homotopy colimits.
    (By the latter we mean that their induced morphism of derivators is cocontinuous.)\label{item:immc1}
  \item The underlying pseudofunctor $\bS\op\to\cCAT$ is an indexed (symmetric) monoidal category with indexed homotopy coproducts preserved by $\ten$.\label{item:immc2}
  \item The external product functors
    \[ \boxtimes \colon \sC^R \times\sC^S \to \sC^{R\times S} \]
    are left Quillen adjoints of two variables.\label{item:immc3}
  \item For any cofibrant $M\in\sC^R$ and any cofibrant replacement $Q \lS_\star \to \lS_\star$ of the unit $\lS_\star$ of $\sC_\star$, the induced maps
    \begin{align*}
      Q \lS_\star \boxtimes M &\to \lS_\star\boxtimes M \cong M \qquad\text{and}\\
      M\boxtimes Q \lS_\star &\to M \boxtimes\lS_\star \cong M
    \end{align*}
    are isomorphisms.\label{item:immc4}
  \item The derived transformations of the isomorphisms
    \[ \pb{f} M \boxtimes \pb{g} N \cong \pb{(f\times g)}(M\boxtimes N) \]
    are again isomorphisms.\label{item:immc5}
  \item The derived transformations of the Beck-Chevalley isomorphisms for any homotopy pullback square in $\bS$ are again isomorphisms.\label{item:immc6}
  \item The derived functor $(\mathbf{R}\pb{f})_A\colon \ho(\sC^S)(A) \to \ho(\sC^R)(A)$ has a right adjoint $(\copf{f})_A$ for all $A\in\cCat$.\label{item:immc7}
  \end{enumerate}
\end{defn}

Condition~\ref{item:immc4} may look familiar from the theory of monoidal model categories, but it is actually also a simple example of a derived natural transformation as in~\cite{shulman:dblderived}.
Note also that we do not assume \sC itself to have indexed products; it may, but all we need for the following theorem is condition~\ref{item:immc7}.
The primary example of parametrized spectra does have indexed products on the point-set level, but it is unclear in general whether or how these are related directly to the derived ones, which are constructed using Brown representability.

\begin{thm}\label{thm:homotopy-indexed-monoidal}
  If $\sC$ is an $\bS$-indexed (symmetric) monoidal model category, then we have an $\bS$-indexed closed (symmetric) monoidal derivator $\ho(\sC)$, defined by $\ho(\sC)^R = \ho(\sC^R)$, which has indexed homotopy products and indexed homotopy coproducts which are preserved by $\ten$.
\end{thm}
\begin{proof}
  Essentially by assumption on the morphisms in $\cRMODEL_{\mathrm{cc}}$, passage to homotopy categories and right derived functors induces a pseudofunctor $\ho(\sC)\colon \bS\op \to \cDER_{\mathrm{cc}}$.
  Now it is shown in~\cite{shulman:frbi} that a (symmetric) monoidal structure on an indexed category can equivalently be described by its external product functors (there represented in terms of a monoidal structure on the corresponding categorical fibration), and this can be extended levelwise to derivators.
  In our case, \autoref{def:immc}\ref{item:immc3} ensures that the external product induces two-variable morphisms of derivators, while~\ref{item:immc5} ensures that these are appropriately pseudonatural.
  The associativity isomorphism (and the symmetry isomorphism, if present) descends automatically, since all functors involved are left Quillen, while assumption~\ref{item:immc4} ensures that the unit isomorphism descends.
  Coherence for all of these isomorphisms follows from the functoriality of derived transformations.
  Thus, $\ho(\sC)$ is an indexed (symmetric) monoidal derivator.

  For indexed homotopy coproducts, the Quillen adjunctions $\pf{f}\dashv \pb{f}$ descend to adjunctions of derivators, and assumption~\ref{item:immc6} ensures that the Beck-Chevalley condition for homotopy pullback squares descends.
  By \autoref{thm:altproj}, for preservation by $\ten$ it suffices to check that the isomorphism $\pf{(\id_C\times f)}(M\boxtimes N) \to M\boxtimes \pf{f}N$ descends to a derived isomorphism; but this is automatic because all functors involved are left Quillen.

  Now since each $\bR\pb{f}$ is cocontinuous, by~\cite[Proposition~2.9]{groth:ptstab} the adjoints $(\copf{f})_A$ assumed by~\ref{item:immc7} assemble into a morphism of derivators right adjoint to $\bR\pb{f}$.
  Thus, by the remark after \autoref{def:indexed-products}, $\ho(\sC)$ has indexed homotopy products as well.
  Finally, closedness follows from \autoref{thm:clproj-der}, since by assumption the external products $\boxtimes$ are two-variable Quillen left adjoints, hence induce two-variable left adjoints of derivators.
\end{proof}

The following is the example we care most about, which we used in \S\ref{sec:reidemeister} to describe the Reidemeister trace.

\begin{thm}\label{thm:parametrized-spectra}
  Let \bS be the category of compactly generated topological spaces.
  Then there is an \bS-indexed closed symmetric monoidal derivator $\ho(\cSp)$ with indexed homotopy products and indexed homotopy coproducts preserved by $\ten$, where $\ho(\cSp)^R$ is the homotopy category of $R$-parametrized spectra.
\end{thm}
\begin{proof}
  We apply \autoref{thm:homotopy-indexed-monoidal}.
  For $R\in\bS$, let $\cSp^R$ be the category of parametrized orthogonal spectra over $R$.
  By~\cite[Theorem 11.4.1 and Proposition 11.4.8]{maysig:pht}, these form an \bS-indexed closed symmetric monoidal category with indexed products and indexed coproducts preserved by $\ten$.

  Now by~\cite[Theorem~12.3.10]{maysig:pht}, each $\cSp^R$ has a stable model structure, for which each adjunction $\pf{f} \colon \cSp^R \toot \cSp^S \colon \pb{f}$ is Quillen by~\cite[Proposition~12.6.7]{maysig:pht}.
  Condition~\ref{item:immc3} of \autoref{def:immc} holds by~\cite[Prop~12.6.5]{maysig:pht}, while~\ref{item:immc4} holds because $\lS_\star$ is cofibrant (note that $\cSp^\star$ coincides with the monoidal model category of ordinary orthogonal spectra from~\cite{mmss:mcds}).
  Conditions~\ref{item:immc5} and~\ref{item:immc6} are shown in~\cite[Theorem~13.7.2]{maysig:pht}. and~\cite[Theorem~13.7.7]{maysig:pht} respectively.
  
  It remains to show that each $\mathbf{R}\pb{f}$ is cocontinuous and has a levelwise right adjoint.
  The proof of~\cite[Theorem~13.1.18]{maysig:pht} shows that $\mathbf{R}\pb{f}$ preserves coproducts.
  Since it is a right adjoint, it preserves pullbacks; and since its domain and codomain are stable, this implies that it preserves pushouts as well.
  Thus, by \cite[Theorem 7.13]{PS5} %\cite[\autoref{PS5:thm:colim-constr}]{PS5} 
  it is cocontinuous.

  Finally, for the existence of a right adjoint to $(\mathbf{R}\pb{f})_A$, we invoke Brown representability following~\cite[Theorem~13.1.18]{maysig:pht} (which is the case $A=\bbone$).
  Since (as we have just shown) $\mathbf{R}\pb{f}$ is a cocontinuous morphism of derivators, $(\mathbf{R}\pb{f})_A$ is an exact and coproduct-preserving functor of triangulated categories.
  Thus, by~\cite[Theorem~13.1.17]{maysig:pht} it suffices to show that each triangulated category $\ho(\cSp^R)(A)$ is compactly generated.
  This is shown in the case $A=\bbone$ by~\cite[Lemma~13.1.11]{maysig:pht}; the generating set was denoted $\sD_R$.

  For general $A\in\cCat$, recall that the ``evaluate at $a$'' functors $\ev{a}\colon \ho(\cSp^R)(A) \to \ho(\cSp^R)(\bbone)$ are jointly conservative.
  Thus, if $\lan{a}$ denotes the left adjoint of $\ev{a}$ (i.e.\ homotopy left Kan extension), then $\bigcup_{a} \lan{a}(\sD_R)$ is a generating set for $\ho(\cSp^R)(A)$.
  And since $\ev{a}$ preserves coproducts (being a left adjoint), $\lan{a}$ preserves compact objects~\cite[Remark~13.1.9]{maysig:pht}.
  Thus, $\ho(\cSp^R)(A)$ is compactly generated.
\end{proof}

Therefore, by \autoref{thm:der-bicat-indexed} and \autoref{thm:der-bicat-indexed-closed} we have:

\begin{thm}\label{thm:exbicat}
  There is a closed derivator bicategory \Ex, with a shadow valued in the homotopy derivator of spectra.
  Its objects are compactly generated spaces, its hom $\Ex(R,S)$ is the homotopy derivator of the model category of spectra parametrized over $R\times S$, and its underlying ordinary bicategory is the one constructed in~\cite[Chapter~17]{maysig:pht}.
\end{thm}

\bibliographystyle{alpha}
\bibliography{additivity}

\end{document}

%% file: decls.tex
\usepackage{amssymb,amsmath,stmaryrd,mathrsfs}
\def\definetac{\newif\iftac}    % Can't define a \newif inside another \if!
\ifx\tactrue\undefined
  \definetac
  \ifx\state\undefined\tacfalse\else\tactrue\fi
\fi
\iftac\else\usepackage{amsthm}\fi
\usepackage[all,2cell]{xy}
\UseAllTwocells
\usepackage{enumitem}
\usepackage{color}
\definecolor{darkgreen}{rgb}{0,0.45,0} 
\usepackage[pagebackref,colorlinks,citecolor=darkgreen,linkcolor=darkgreen]{hyperref}
\usepackage{mathtools}          % for all sorts of things
\usepackage{braket}             % for \Set, etc.

\usepackage{url}                % for citations to web sites
\usepackage{xspace}             % put spaces after a \command in text

\makeatletter
\let\ea\expandafter

\def\mdef#1#2{\ea\ea\ea\gdef\ea\ea\noexpand#1\ea{\ea\ensuremath\ea{#2}\xspace}}
\def\alwaysmath#1{\ea\ea\ea\global\ea\ea\ea\let\ea\ea\csname your@#1\endcsname\csname #1\endcsname
  \ea\def\csname #1\endcsname{\ensuremath{\csname your@#1\endcsname}\xspace}}

\DeclareRobustCommand\widecheck[1]{{\mathpalette\@widecheck{#1}}}
\def\@widecheck#1#2{%
    \setbox\z@\hbox{\m@th$#1#2$}%
    \setbox\tw@\hbox{\m@th$#1%
       \widehat{%
          \vrule\@width\z@\@height\ht\z@
          \vrule\@height\z@\@width\wd\z@}$}%
    \dp\tw@-\ht\z@
    \@tempdima\ht\z@ \advance\@tempdima2\ht\tw@ \divide\@tempdima\thr@@
    \setbox\tw@\hbox{%
       \raise\@tempdima\hbox{\scalebox{1}[-1]{\lower\@tempdima\box
\tw@}}}%
    {\ooalign{\box\tw@ \cr \box\z@}}}

\newcount\foreachcount

\def\foreachletter#1#2#3{\foreachcount=#1
  \ea\loop\ea\ea\ea#3\@alph\foreachcount
  \advance\foreachcount by 1
  \ifnum\foreachcount<#2\repeat}

\def\foreachLetter#1#2#3{\foreachcount=#1
  \ea\loop\ea\ea\ea#3\@Alph\foreachcount
  \advance\foreachcount by 1
  \ifnum\foreachcount<#2\repeat}

\def\definescr#1{\ea\gdef\csname s#1\endcsname{\ensuremath{\mathscr{#1}}\xspace}}
\foreachLetter{1}{27}{\definescr}
\def\definecal#1{\ea\gdef\csname c#1\endcsname{\ensuremath{\mathcal{#1}}\xspace}}
\foreachLetter{1}{27}{\definecal}
\def\definebold#1{\ea\gdef\csname b#1\endcsname{\ensuremath{\mathbf{#1}}\xspace}}
\foreachLetter{1}{27}{\definebold}
\def\definebb#1{\ea\gdef\csname l#1\endcsname{\ensuremath{\mathbb{#1}}\xspace}}
\foreachLetter{1}{27}{\definebb}
\def\definefrak#1{\ea\gdef\csname f#1\endcsname{\ensuremath{\mathfrak{#1}}\xspace}}
\foreachletter{1}{9}{\definefrak} % \fi is something else!
\foreachletter{10}{27}{\definefrak}
\foreachLetter{1}{27}{\definefrak}
\def\definebar#1{\ea\gdef\csname #1bar\endcsname{\ensuremath{\overline{#1}}\xspace}}
\foreachLetter{1}{27}{\definebar}
\foreachletter{1}{8}{\definebar} % \hbar is something else!
\foreachletter{9}{15}{\definebar} % \obar is something else!
\foreachletter{16}{27}{\definebar}
\def\definetil#1{\ea\gdef\csname #1til\endcsname{\ensuremath{\widetilde{#1}}\xspace}}
\foreachLetter{1}{27}{\definetil}
\foreachletter{1}{27}{\definetil}
\def\definehat#1{\ea\gdef\csname #1hat\endcsname{\ensuremath{\widehat{#1}}\xspace}}
\foreachLetter{1}{27}{\definehat}
\foreachletter{1}{27}{\definehat}
\def\definechk#1{\ea\gdef\csname #1chk\endcsname{\ensuremath{\widecheck{#1}}\xspace}}
\foreachLetter{1}{27}{\definechk}
\foreachletter{1}{27}{\definechk}
\def\defineul#1{\ea\gdef\csname u#1\endcsname{\ensuremath{\underline{#1}}\xspace}}
\foreachLetter{1}{27}{\defineul}
\foreachletter{1}{27}{\defineul}

\def\autofmt@n#1\autofmt@end{\mathrm{#1}}
\def\autofmt@b#1\autofmt@end{\mathbf{#1}}
\def\autofmt@l#1#2\autofmt@end{\mathbb{#1}\mathsf{#2}}
\def\autofmt@c#1#2\autofmt@end{\mathcal{#1}\mathit{#2}}
\def\autofmt@s#1#2\autofmt@end{\mathscr{#1}\mathit{#2}}
\def\autofmt@f#1\autofmt@end{\mathfrak{#1}}
\def\autofmt@u#1\autofmt@end{\underline{\smash{\mathsf{#1}}}}
\def\autofmt@U#1\autofmt@end{\underline{\underline{\smash{\mathsf{#1}}}}}
\def\autofmt@h#1\autofmt@end{\widehat{#1}}
\def\autofmt@r#1\autofmt@end{\overline{#1}}
\def\autofmt@t#1\autofmt@end{\widetilde{#1}}
\def\autofmt@k#1\autofmt@end{\check{#1}}

\def\auto@drop#1{}
\def\autodef#1{\ea\ea\ea\@autodef\ea\ea\ea#1\ea\auto@drop\string#1\autodef@end}
\def\@autodef#1#2#3\autodef@end{%
  \ea\def\ea#1\ea{\ea\ensuremath\ea{\csname autofmt@#2\endcsname#3\autofmt@end}\xspace}}
\def\autodefs@end{blarg!}
\def\autodefs#1{\@autodefs#1\autodefs@end}
\def\@autodefs#1{\ifx#1\autodefs@end%
  \def\autodefs@next{}%
  \else%
  \def\autodefs@next{\autodef#1\@autodefs}%
  \fi\autodefs@next}

\DeclareSymbolFont{bbold}{U}{bbold}{m}{n}
\DeclareSymbolFontAlphabet{\mathbbb}{bbold}

\newcommand{\bbone}{\ensuremath{\mathbbb{1}}\xspace}

\mdef\delbar{\overline{\partial}}
\let\sm\wedge

\mdef\hf{\textstyle\frac12 }
\mdef\thrd{\textstyle\frac13 }
\mdef\qtr{\textstyle\frac14 }

\newcommand{\op}{^{\mathrm{op}}}

\SelectTips{cm}{}
\newdir{ >}{{}*!/-10pt/@{>}}    % extra spacing for tail arrows in XYpic

\let\iso\cong

\mdef\Id{\mathrm{Id}}
\mdef\id{\mathrm{id}}
\alwaysmath{ell}
\alwaysmath{infty}
\alwaysmath{odot}
\def\frc#1/#2.{\frac{#1}{#2}}   % \frc x^2+1 / x^2-1 .
\mdef\ten{\mathrel{\otimes}}

\mdef\sqten{\mathrel{\boxtimes}}

\DeclareRobustCommand\widecheck[1]{{\mathpalette\@widecheck{#1}}}
\def\@widecheck#1#2{%
    \setbox\z@\hbox{\m@th$#1#2$}%
    \setbox\tw@\hbox{\m@th$#1%
       \widehat{%
          \vrule\@width\z@\@height\ht\z@
          \vrule\@height\z@\@width\wd\z@}$}%
    \dp\tw@-\ht\z@
    \@tempdima\ht\z@ \advance\@tempdima2\ht\tw@ \divide\@tempdima\thr@@
    \setbox\tw@\hbox{%
       \raise\@tempdima\hbox{\scalebox{1}[-1]{\lower\@tempdima\box
\tw@}}}%
    {\ooalign{\box\tw@ \cr \box\z@}}}

\DeclareMathOperator\lan{Lan}

\DeclareMathOperator\colim{colim}

\newcommand{\too}[1][]{\ensuremath{\overset{#1}{\longrightarrow}}}

\let\toot\rightleftarrows

\mdef\we{\overset{\sim}{\longrightarrow}}
\mdef\leftwe{\overset{\sim}{\longleftarrow}}

\let\maps\colon
\newcommand{\fib}{\mathsf{fib}}

\let\xto\xrightarrow

\def\rightarrowtailfill@{\arrowfill@{\Yright\joinrel\relbar}\relbar\rightarrow}
\newcommand\xrightarrowtail[2][]{\ext@arrow 0055{\rightarrowtailfill@}{#1}{#2}}

\def\twoheadrightarrowfill@{\arrowfill@{\relbar\joinrel\relbar}\relbar\twoheadrightarrow}
\newcommand\xtwoheadrightarrow[2][]{\ext@arrow 0055{\twoheadrightarrowfill@}{#1}{#2}}

\def\slashedarrowfill@#1#2#3#4#5{%
  $\m@th\thickmuskip0mu\medmuskip\thickmuskip\thinmuskip\thickmuskip
   \relax#5#1\mkern-7mu%
   \cleaders\hbox{$#5\mkern-2mu#2\mkern-2mu$}\hfill
   \mathclap{#3}\mathclap{#2}%
   \cleaders\hbox{$#5\mkern-2mu#2\mkern-2mu$}\hfill
   \mkern-7mu#4$%
}
\def\rightslashedarrowfill@{%
  \slashedarrowfill@\relbar\relbar\mapstochar\rightarrow}
\newcommand\xslashedrightarrow[2][]{%
  \ext@arrow 0055{\rightslashedarrowfill@}{#1}{#2}}
\mdef\hto{\xslashedrightarrow{}}
\mdef\htoo{\xslashedrightarrow{\quad}}

\def\shvar#1#2{{\ensuremath{%
  \hspace{1mm}\makebox[-1mm]{$#1\langle$}\makebox[0mm]{$#1\langle$}\hspace{1mm}%
  {#2}%
  \makebox[1mm]{$#1\rangle$}\makebox[0mm]{$#1\rangle$}%
}}}
\def\sh{\shvar{}}

\long\def\my@drawfill#1#2;{%
\@skipfalse
\fill[#1,draw=none] #2;
\@skiptrue
\draw[#1,fill=none] #2;
}
\newif\if@skip
\newcommand{\skipit}[1]{\if@skip\else#1\fi}
\newcommand{\drawfill}[1][]{\my@drawfill{#1}}

\newif\ifhyperref
\@ifpackageloaded{hyperref}{\hyperreftrue}{\hyperreffalse}
\iftac
  \let\your@state\state
  \def\state#1{\my@state#1}
  \def\my@state#1.{\gdef\currthmtype{#1}\your@state{#1.}}
  \let\your@staterm\staterm
  \def\staterm#1{\my@staterm#1}
  \def\my@staterm#1.{\gdef\currthmtype{#1}\your@staterm{#1.}}
  \let\defthm\newtheorem
  \def\switchtotheoremrm{\let\defthm\newtheoremrm}
  \def\currthmtype{}
  \ifhyperref
    \def\autoref#1{\ref*{label@name@#1}~\ref{#1}}
  \else
    \def\autoref#1{\ref{label@name@#1}~\ref{#1}}
  \fi
  \AtBeginDocument{%
    \let\old@label\label%
    \def\label#1{%
      {\let\your@currentlabel\@currentlabel%
        \edef\@currentlabel{\currthmtype}%
        \old@label{label@name@#1}}%
      \old@label{#1}}
  }
\else
  \ifhyperref
    \def\defthm#1#2{%
      \newtheorem{#1}{#2}[section]%
      \expandafter\def\csname #1autorefname\endcsname{#2}%
      \expandafter\let\csname c@#1\endcsname\c@thm}
  \else
    \def\defthm#1#2{\newtheorem{#1}[thm]{#2}}
    \ifx\SK@label\undefined\let\SK@label\label\fi
    \let\old@label\label
    \let\your@thm\@thm
    \def\@thm#1#2#3{\gdef\currthmtype{#3}\your@thm{#1}{#2}{#3}}
    \def\currthmtype{}
    \def\label#1{{\let\your@currentlabel\@currentlabel\def\@currentlabel%
        {\currthmtype~\your@currentlabel}%
        \SK@label{#1@}}\old@label{#1}}
    \def\autoref#1{\ref{#1@}}
  \fi
\fi

\newtheorem{thm}{Theorem}[section]

\defthm{cor}{Corollary}
\defthm{prop}{Proposition}
\defthm{lem}{Lemma}
\defthm{sch}{Scholium}
\defthm{assume}{Assumption}
\defthm{claim}{Claim}
\defthm{conj}{Conjecture}
\defthm{hyp}{Hypothesis}
\defthm{fact}{Fact}
\iftac\switchtotheoremrm\else\theoremstyle{definition}\fi
\defthm{defn}{Definition}
\defthm{notn}{Notation}
\iftac\switchtotheoremrm\else\theoremstyle{remark}\fi
\defthm{rmk}{Remark}
\defthm{eg}{Example}
\defthm{egs}{Examples}
\defthm{ex}{Exercise}
\defthm{ceg}{Counterexample}

\def\thmqedhere{\expandafter\csname\csname @currenvir\endcsname @qed\endcsname}

\setitemize[1]{leftmargin=2em}
\setenumerate[1]{leftmargin=*}

\iftac
  \let\c@equation\c@subsection
\else
  \let\c@equation\c@thm
\fi
\numberwithin{equation}{section}

\@ifpackageloaded{mathtools}{\mathtoolsset{showonlyrefs,showmanualtags}}{}

\alwaysmath{alpha}
\alwaysmath{beta}
\alwaysmath{gamma}
\alwaysmath{Gamma}
\alwaysmath{delta}
\alwaysmath{Delta}
\alwaysmath{epsilon}
\mdef\ep{\varepsilon}
\alwaysmath{zeta}
\alwaysmath{eta}
\alwaysmath{theta}
\alwaysmath{Theta}
\alwaysmath{iota}
\alwaysmath{kappa}
\alwaysmath{lambda}
\alwaysmath{Lambda}
\alwaysmath{mu}
\alwaysmath{nu}
\alwaysmath{xi}
\alwaysmath{pi}
\alwaysmath{rho}
\alwaysmath{sigma}
\alwaysmath{Sigma}
\alwaysmath{tau}
\alwaysmath{upsilon}
\alwaysmath{Upsilon}
\alwaysmath{phi}
\alwaysmath{Pi}
\alwaysmath{Phi}
\mdef\ph{\varphi}
\alwaysmath{chi}
\alwaysmath{psi}
\alwaysmath{Psi}
\alwaysmath{omega}
\alwaysmath{Omega}

\makeatother